\documentclass[11pt]{amsart}
\usepackage{amsmath}

\usepackage[margin=3cm]{geometry}
\usepackage{amscd,amsfonts,amssymb}
\usepackage{enumitem}
\usepackage{mathrsfs}
\usepackage{amsthm}
\usepackage{makecell}
\newenvironment{definition}[1][Definition:]{\begin{trivlist}
\item[\hskip \labelsep {\bfseries #1}]}{\end{trivlist}}

\newtheorem{theorem}{Theorem}[section]
\newtheorem{lemma}[theorem]{Lemma}
\newtheorem{proposition}[theorem]{Proposition}
\newtheorem{corollary}[theorem]{Corollary}

\newenvironment{remark}[1][Remark:]{\begin{trivlist}
\item[\hskip \labelsep {\bfseries #1}]}{\end{trivlist}}

\newcommand{\C}{\mathbb{C}}
\newcommand{\D}{\mathcal{D}}
\newcommand{\R}{\mathbb{R}}
\newcommand{\Hcal}{\mathcal{H}}

\newcommand{\im}{\mathrm{im} \,}

\newcommand{\vol}{\ensuremath{\mathrm{vol}}}

\newcommand{\T}{\mathcal{T}}

\newcommand{\X}{\mathcal{X}}
\newcommand{\Y}{\mathcal{Y}}

\newcommand{\M}{\mathcal{M}}
\newcommand{\N}{\mathcal{N}}

\newcommand{\eL}{\mathcal{L}}

\renewcommand{\tilde}{\widetilde}
\renewcommand{\Im}{\mathrm{Im}\,}

\renewcommand{\Re}{\mathrm{Re}}

\newcommand{\Emb}{\mathrm{Emb}}
\newcommand{\Diff}{\D}
\newcommand{\U}{\mathcal{U}}

\newcommand{\Herm}{\mathrm{Herm}}
\newcommand{\SL}{\mathrm{SL}}
\newcommand{\PSL}{\mathrm{PSL}}

\newcommand{\two}{\mathrm{\MakeUppercase{\romannumeral 2}}}

\newcommand{\norm}[1]{\lVert #1 \rVert}

\newcommand{\supp}{\mathrm{supp} \,}
\newcommand{\ssubset}{\subset\joinrel\subset}

\newcommand{\news}{s}
\newcommand{\newszero}{\news_{\Lambda^0(L)}}
\newcommand{\newstwo}{\news_{\Lambda^2(L)}}
\newcommand{\lot}{\mathrm{l.o.t.}}

\newcommand{\reg}{\mathrm{reg}}

\newcommand{\B}{\mathcal{B}}

\DeclareMathOperator{\coker}{\mathrm{coker}}

\DeclareMathOperator{\ind}{\mathrm{ind}}
\newcommand{\ed}{\mathrm{d}}

\begin{document}
\title[Generic special Lagrangian moduli spaces]{Generic special Lagrangian moduli spaces of a non-K\"ahler Calabi--Yau threefold}
\author{Benjamin Friedman}
\address{Centre interuniversitaire de recherches en géométrie et topologie (CIRGET) \\
                 Universit\'e du Qu\'ebec \`a Montr\'eal \\
                 Montr\'eal, QC H2X 3Y7}
     \email{friedman.benjamin@uqam.ca }
\maketitle

\begin{abstract}
    Given a (possibly non-K\"ahler) Calabi--Yau threefold $(X,\Omega)$, we introduce the notion of a (perturbed) special Lagrangian (SL) submanifold of $(X,\omega,\Omega)$, where $\omega$ is a Hermitian metric on $X$. The equations defining this class of submanifolds reduce to the usual SL equations when $\omega$ is a K\"ahler metric. Using the Sard--Smale technique, we prove the existence of a comeagre set of Hermitian metrics $\omega$ on $X$ such that the moduli space of perturbed SL submanifolds in $(X,\omega,\Omega)$ consists of isolated points.
\end{abstract}

%
%
%
%
%
%
%
%
%

%
%
%
    \section{Introduction}
    \subsection{Motivation and main results}
    
        Calabi--Yau threefolds lie at the intersection of several disciplines, being of interest to differential geometers, algebraic geometers, and string theorists. However, discourses across these fields are muddled at times by the fact that several inequivalent definitions of a Calabi--Yau threefold are found throughout the literature. For this reason, let us define from the outset precisely what we mean:
        
    \begin{definition}
    A \textbf{Calabi--Yau threefold} is a compact, complex manifold $X$ of complex dimension $3$ which admits a nowhere vanishing holomorphic $(3,0)$-form $\Omega$.
    \end{definition}
    
    Note that the existence of such a form $\Omega$ is equivalent to $X$ having a trivial canonical bundle. By contrast, the traditional definition in differential geometry is that a Calabi--Yau threefold is a compact K\"ahler manifold of complex dimension $3$ with vanishing first Chern class. These manifolds admit Ricci flat K\"ahler metrics by Yau's Theorem \cite{Yau78}. In recent years there has been increasing interest in Calabi--Yau threefolds according to our chosen definition, which are sometimes referred to as \emph{non-K\"ahler Calabi--Yau threefolds}.
    
    One motivation for the study of non-K\"ahler Calabi--Yau threefolds comes from algebraic geometry \cite{Friedman_R}, \cite{Reid87}, where such spaces may arise from conifold transitions as the smoothings $X_t$ of a singular conifold $X_0$ obtained by contracting $(-1,-1)$-curves in an initial K\"ahler Calabi--Yau threefold $\widehat{X}$. Another motivation comes from considerations of supersymmetry \cite{Strominger86} in string theory.
    
    Given a Calabi--Yau threefold $(X,\Omega)$ equipped with a Hermitian metric $\omega$, a \textbf{special Lagrangian submanifold} of $(X,\omega,\Omega)$ with \textbf{phase} $e^{i\theta_0} \in S^1$ is an embedded submanifold $L^3$ in $X$ satisfying
\begin{align*}
& \left.\omega\right|_L = 0 \\
& \left.\Im e^{-i\theta_0}\Omega\right|_L =0.
\end{align*}

 Special Lagrangian submanifolds were first introduced by Harvey and Lawson \cite{HarveyLawson}, who proved that after a conformal change of the metric to ensure that $|\Omega|_{\omega} \equiv 1$, such a submanifold is calibrated by the closed $3$-form $\Re\; e^{-i\theta_0} \Omega$, and hence minimizes the conformally transformed area functional in its homology class. In general, a Lagrangian submanifold $L$ in $X$ is associated to a phase function $e^{i\theta}: L \to S^1$ where $\left.\Omega\right|_L = e^{i\theta} |\Omega|_{\omega} \ed\vol_L$, and the \emph{special} Lagrangian submanifolds are exactly those with constant phase $e^{i\theta_0}$. For our purposes the phase is irrelevant, so we will consider only special Lagrangian submanifolds for which $e^{i\theta_0} = 1$. In the case of K\"ahler Ricci-flat metric, one can show by the maximum principle that $|\Omega|_{\omega}$ is constant without the need for a conformal change, so that special Lagrangian submanifolds in that setting are calibrated.

	We will let $\M^{\SL}(X,\omega,\Omega,L)$ denote the \emph{moduli space} of embedded special Lagrangian submanifolds of $(X,\omega,\Omega)$ which are diffeomorphic to a given $L$.
		
	In the K\"ahler case, McLean \cite{McLean} showed that $\M^{\SL}(X,\omega,\Omega,L)$ is itself a manifold of dimension $b_1(L)$. The idea is that at a particular special Lagrangian $f: L \hookrightarrow (X,\omega,\Omega)$, the condition $f^* \omega = 0$ together with the induced Riemannian metric on $L$ gives an isomorphism between normal vector fields along $f$ and $1$-forms on $L$. McLean showed that the infinitesimal special Lagrangian deformations correspond exactly to harmonic $1$-forms under this isomorphism, and that these deformations integrate to a $b_1(L)$-dimensional family of nearby special Lagrangian submanifolds. See Marshall's doctoral dissertation \cite{Marshall} for a particularly clear treatment.
	
	In a recent paper \cite{CGPY23}, Collins, Gukov, Picard, and Yau consider the problem of special Lagrangian submanifolds in a non-K\"ahler Calabi--Yau manifold. They construct unobstructed families of special Lagrangian submanifolds in manifolds of Iwasawa type, and prove that for small $t$, the vanishing spheres on the smoothings $X_t$ of a conifold transition can be deformed to special Lagrangian $3$-spheres with respect to $(X_t,\omega_t,\Omega_t)$, where $\omega_t$ denotes the balanced metric on $X_t$ constructed by Fu, Li, and Yau \cite{FLY12}.
	
	However, the work in \cite{CGPY23} leaves the moduli space problem open, and while they compute the linearization of the special Lagrangian equations, the authors point out that it is unclear whether an analogue of McLean's theorem holds in the non-K\"ahler setting. The aim of this paper is to address this gap in our understanding. 
	
	In order to investigate the moduli space problem, we work with the perturbed equations
    \begin{align*}
    &\left.\omega\right|_L + \ed^{\dagger} \rho = 0 \\
    &\left.\Im \Omega\right|_L = 0
    \end{align*}
    where $\rho$ is a small exact $3$-form on $L$, and $\ed^{\dagger}$ is the $L^2$-adjoint of the exterior derivative $\ed$ on $L$ with respect to the induced metric. See Section~\ref{subsec-sl-psl} below for a more thorough definition and discussion of the system. When $L$ is compact, this system reduces to the usual special Lagrangian equations in the K\"ahler case, and is analytically better behaved than the usual SL system in the non-K\"ahler case. Letting $\M^{\PSL}(X,\omega,\Omega,L)$ denote the moduli space of embedded solutions to the perturbed problem (as we will see in Section~\ref{subsec-moduli-space}, a solution includes both an embedded submanifold $L$ and the $\Diff(L)$-orbit of an exact $3$-form $\rho$ on $L$ as data), our main result is the following:
	\begin{theorem}\label{thm-main}
    Let $(X,\Omega)$ be a compact Calabi--Yau threefold, and let $L^3$ be a fixed compact manifold of dimension $3$. Then within the space $\Herm(X)$ of Hermitian metrics on $X$ with the $C^{\infty}$ topology, there is a comeagre set $\Herm_{\reg}(X)$ such that for all $\omega \in \Herm_{\reg}(X)$, the moduli space $\M^{\PSL}(X,\omega,\Omega,L)$ consists of isolated points.
    \end{theorem}
%
%
	Note that Theorem~\ref{thm-main} also implies that submanifolds of $X$ which are special Lagrangian in the usual sense are also isolated, since these submanifolds are exactly the PSL submanifolds for which $\rho = 0$.
	
	               \begin{table}[t]
    \begin{center}
    \begin{tabular}{|l|l|l|}
    \hline
     & \thead{Transversality\\of pseudoholomorphic curves} & \thead{Theorem~\ref{thm-main}} \\
     \hline
     \thead{Fixed\\ambient space:} &  \thead{Symplectic manifold $(X,\omega)$} & \thead{Non-K\"ahler Calabi--Yau\\threefold $(X,J,\Omega)$} \\
     \hline
\thead{Maps under\\consideration:} & \thead{$J$-holomorphic curves \\$u: (\Sigma, j) \to (X,J,\omega)$\\} & \thead{PSL submanifolds \\$f: (L, \rho) \to (X, J, \omega, \Omega)$}\\
\hline
     \thead{Class of\\regular objects:} & \thead{Generic almost complex\\structures $J$ compatible with $\omega$} & \thead{Generic metrics $\omega$ that are\\Hermitian with respect to $J$}\\
     \hline
    \end{tabular}
    \end{center}
    
    \caption{In the classical transversality theory for pseudoholomorphic curves, the symplectic form $\omega$ is fixed while $J$ is chosen generically. In this paper, the complex structure $J$ is fixed while $\omega$ is chosen generically.}\label{table-compare}
    \end{table}

	Theorem~\ref{thm-main} is in contrast to McLean's Theorem. Indeed, while the latter states that in a K\"ahler Calabi--Yau threefold the special Lagrangian submanifolds come in $b_1(L)$-dimensional families, our theorem implies that after perturbing the K\"ahler metric to a generic nearby Hermitian metric, all of the special Lagrangian submanifolds with respect to the new metric are isolated. Note that while the metric is perturbed, the complex structure of $X$ is fixed, so that Theorem~\ref{thm-main} is capturing a property of the underlying complex manifold.
    
    Theorem~\ref{thm-main} is also surprising when considered in the context of the SYZ conjecture, first proposed in 1996 by Strominger, Yau, and Zaslow \cite{SYZ96}. See Chapter 9 of \cite{Joyce07} for an introduction to SYZ from a geometric perspective. The SYZ conjecture posits that any compact K\"ahler Calabi--Yau threefold $(X,\omega,\Omega)$ is the total space of a fibration $\pi: X \to B$, where $B$ is a compact manifold of real dimension $3$, and the fibres $\pi^{-1}(b)$ are special Lagrangian submanifolds of $(X,\omega,\Omega)$, regular for all $b$ in a dense open subset of $B$. See \cite{Li22} for recent developments on constructions of such fibrations. If this conjecture is true, then Theorem~\ref{thm-main} implies that the special Lagrangian fibration is unstable under non-K\"ahler deformations of the metric.

    The method by which we prove Theorem~\ref{thm-main} is similar to the proof that the moduli spaces of $J$-holomorphic curves on a symplectic manifold $X$ are finite-dimensional manifolds for a comeagre set of almost complex structures $J$ on $X$ (see Chapter~3 of \cite{McDuffSalamon}). See Table~\ref{table-compare} for a comparison between our work and the situation for $J$-holomorphic curves. This Sard--Smale technique was also recently used to great effect in the work of Pardon \cite{Pardon} in enumerative geometry.
    
        Continuing the analogy with $J$-holomorphic curves, a potential application of Theorem~\ref{thm-main} is in the development of topological invariants -- analogous to the Gromov--Witten invariant -- associated to a non-K\"ahler Calabi--Yau threefold $(X,\Omega)$ which involve counting isolated special Lagrangian submanifolds. This idea is inspired by the program of Joyce \cite{Joyce02} concerning invariants defined by counting special Lagrangian \emph{rational homology 3-spheres}, which appear in 0-dimensional moduli spaces. In contrast, our Theorem~\ref{thm-main} produces isolated special Lagrangian submanifolds of any topological type. 
    
    One can compare Theorem~\ref{thm-main} to the remarkable recent result of Windes \cite{Windes}, a transversality result for solutions of a different system of perturbed special Lagrangian equations. This system is distinct from the system (\ref{eqns-SLag-pert}) in this paper. The results in \cite{Windes} imply that on a Calabi--Yau threefold, the perturbed equations in \cite{Windes} reduce to the usual special Lagrangian equations, and that the solutions are isolated for a generic choice of $\mathrm{SU}(3)$-structure $(\omega,\Omega)$. A significant difference in our paper is that we take the holomorphic $(3,0)$-form $\Omega$ to be fixed, while the metric $\omega$ varies. Thus, an invariant of a non-K\"ahler Calabi--Yau threefold defined by counting PSL submanifolds would be an invariant of the underlying complex manifold, while an invariant developed using Windes' theory would be an invariant in the category of smooth manifolds.
%
%
%

\subsection{Acknowledgements}

    First and foremost, the author would like to express his gratitude to his advisor S\'ebastien Picard for suggesting this problem and for much valuable guidance along the way. In addition, the author is grateful to his advisor Ailana Fraser, Jim Bryan, Tristan Collins, Patrik Coulibaly, Jooho Lee, Jason Lotay, Emily Autumn Windes, and Peilin Wu for helpful discussions. The author also thanks the anonymous referee for helpful comments on the manuscript.
    
    This material is based upon work supported by the National Science Foundation under Grant No. DMS-1928930, while the author was in residence at the Simons Laufer Mathematical Sciences Institute (formerly MSRI) in Berkeley, California, during the Fall 2024 semester.

\section{PSL submanifolds}

\subsection{The SL and PSL systems}\label{subsec-sl-psl}

Let $(X,J,\Omega)$ be a compact Calabi--Yau threefold, and suppose that $\omega \in \Herm(X)$ is a Hermitian metric on $X$. We say that $f: L^3 \to (X,J,\omega,\Omega)$ is a \textbf{special Lagrangian submanifold} (SL) if
\begin{align}
\begin{split}\label{eqns-SLag}
f^* \omega &= 0 \\
f^* \Im \Omega &= 0,
\end{split}
\end{align}
and we will say that $f: L^3 \to (X,J,\omega,\Omega)$ is a \textbf{perturbed special Lagrangian submanifold} (PSL) if for some $\rho \in \Lambda^3(L)$ with $[\rho] = 0 \in H^3(L)$, we have
\begin{align}
\begin{split}\label{eqns-SLag-pert}
f^* \omega + \ed^{\dagger}_{f^*\omega} \rho &= 0 \\
f^* \Im \Omega &= 0,
\end{split}
\end{align}
where $\ed^{\dagger}_{f^*\omega}$ is the $L^2$ adjoint of $\ed$ with respect to $f^*g$, and $g = \omega(\cdot, J\cdot)$ is the Riemannian metric associated to $\omega$. If the PSL system (\ref{eqns-SLag-pert}) is satisfied, then we say that $f: (L,\rho) \to (X,J,\omega,\Omega)$ is PSL.

\begin{remark} How does the perturbed problem (\ref{eqns-SLag-pert}) relate to the unperturbed special Lagrangian problem (\ref{eqns-SLag})? We note the following observations:
\begin{enumerate}[label=(\roman*)]
\item If $f$ is a solution of (\ref{eqns-SLag-pert}) for a $3$-form $\rho$, where $\rho$ satisfies the additional equation $\ed_{f^*\omega}^{\dagger}\rho = 0$, then $f$ solves (\ref{eqns-SLag}).
\item If $f$ solves (\ref{eqns-SLag-pert}) for some $3$-form $\rho$, where the metric $\omega$ satisfies the K\"ahler condition $\ed \omega=0$ and $L$ is compact, then by (\ref{eqns-SLag-pert}) we have $\ed \ed^{\dagger}_{f^*\omega} \rho = 0$. Therefore
\[
0 = \langle \ed \ed^{\dagger}_{f^*\omega} \rho, \rho \rangle_{L^2} = \langle \ed^{\dagger}_{f^*\omega} \rho, \ed^{\dagger}_{f^*\omega} \rho \rangle_{L^2}.
\]
So $\ed^{\dagger}_{f^*\omega} \rho = 0$, and we conclude that $f$ solves (\ref{eqns-SLag}). Thus, the perturbed problem~(\ref{eqns-SLag-pert}) can be thought of as a non-K\"ahler generalization of the special Lagrangian equations, a generalization that is just as valid in the compact case as (\ref{eqns-SLag}).
\end{enumerate}
\end{remark}

In order to understand the advantage in considering the system (\ref{eqns-SLag-pert}) instead of (\ref{eqns-SLag}), we recall that in the K\"ahler case $\ed\omega=0$, the linearization of (\ref{eqns-SLag}) is given by
\[
\ed\alpha + \ed_{f^*\omega}^{\dagger} \alpha = 0,
\]
where $\alpha$ is a $1$-form on $L$. This linearization is a central ingredient in McLean's Theorem \cite{McLean}, since it implies that infinitesimal SL deformations correspond to harmonic $1$-forms on $L$. However in order to prove that the SL moduli space $\M^{\SL}(X,\omega,\Omega,L)$ is unobstructed, one needs the index of the operator $\ed+\ed_{f^*\omega}^{\dagger}$ to be finite. This works out since $\ed+\ed_{f^*\omega}^{\dagger}$ surjects onto the closed subspace $\ed \Lambda^1(L) \oplus \ed_{f^*\omega}^{\dagger} \Lambda^1(L)$ of the Hodge complex $\Lambda^*(L)$, so one can restrict the codomain of the operator so that the index is finite. However if the K\"ahler condition is relaxed, then the linearization of (\ref{eqns-SLag}) becomes
\[
\ed\alpha + T(\omega) \alpha  + \ed_{f^*\omega}^{\dagger}\left(|\Omega|_{\omega}\alpha\right)=0.
\]
where $T(\omega): \Lambda^1(L) \to \Lambda^2(L)$ is an additional \emph{torsion term} given by $T(\omega)\alpha = f^*(\iota_{J\alpha^{\sharp}} \ed \omega)$ (see Section~2.2 of \cite{CGPY23}). The presence of the torsion term implies that the codomain of the infinitesimal SL operator cannot generally be restricted to any proper closed subspace of $\Lambda^2(L) \oplus \ed_{f^*\omega}^{\dagger} \Lambda^1(L)$. However, if the component of the linearization valued in $2$-forms has an additional term of the form $\ed_{f^*\omega}^{\dagger} \rho$, then the infinitesimal operator can be deformed to the operator $\ed+\ed_{f^*\omega}^{\dagger}$ restricted to the closed subspace $\Lambda^1(L) \oplus \Lambda^3(L)$ of $\Lambda^*(L)$, which is an operator of index $0$. This is proved in Section~\ref{subsec-inf-def}.



\subsection{Totally real geometry}\label{sec-tot-real}

A submanifold $L^3 \subset X$ is said to be \textbf{totally real} if $J(T_p L) \cap T_p L = \{0\}$ for all $p \in L$. This open condition is a natural weakening of the Lagrangian condition $f^*\omega = 0$, a closed condition. Lotay and Pacini \cite{LotayPacini} discuss the relation between these two conditions, and analyze a weak form of the Lagrangian mean curvature flow for totally real submanifolds.

\begin{lemma}\label{lem-tot-real}
Let $\norm{\ed^{\dagger}_{f^*\omega}}$ denote the operator norm of
\[
\ed^{\dagger}_{f^*\omega}: C^{1}\left(\Lambda^3(L)\right) \to C^{0}\left(\Lambda^2(L)\right).
\]
If $\norm{\rho}_{C^{1}} < \norm{\ed^{\dagger}_{f^*\omega}}^{-1}$, then a submanifold $f: L^3 \to (X,J,\omega,\Omega)$ satisfying $f^* \omega + \ed^{\dagger}_{f^*\omega} \rho = 0$ is totally real.
\end{lemma}

\begin{proof}
Let $p \in L$, and fix two unit vectors $u, v \in T_p L$. Let $\theta \in [0,\pi]$ be the smallest angle between $Ju$ and $v$. Then we have
\begin{align*}
|\cos \theta| &= |(f^* \omega)(u,v)| = |(\ed^{\dagger}_{f^*\omega} \rho)(u,v)| \\
&\leq \norm{\ed^{\dagger}_{f^*\omega}} \cdot \norm{\rho}_{C^{1}} < 1.
\end{align*}
Therefore $\theta \in (0,\pi)$. Since $v \in T_pL$ is arbitrary, we see that $Ju \not\in T_pL$.
\end{proof}

Motivated by Lemma~\ref{lem-tot-real}, given $\ell \in \mathbb{Z}_{\geq 0}$ and $a \in [0,1]$ we will define the set
\begin{equation}\label{3-form-tube-def}
\T_{f,\omega}^{\ell,a} := \left\{ \rho \in \Lambda^3(L) : \norm{\rho}_{C^{\ell,a}} < \norm{\ed^{\dagger}_{f^* \omega}}^{-1}\right\}.
\end{equation}

Given a totally real submanifold $f: L^3 \to X$, the fact that
\[
f^* TX = TL \oplus T^{\perp} L = TL \oplus J(TL)
\]
means that the bundles $T^{\perp}L$ and $J(TL)$ are both normal bundles in the topological sense:
\begin{equation}\label{eqn-normal-bundles}
T^{\perp} L \cong J(TL) \cong f^* TX / TL.
\end{equation}

When discussing deformations of submanifolds in Riemannian geometry, usually one considers infinitesimal normal variations along the submanifold. But since both $T^{\perp}L$ and $J(TL)$ are ``normal bundles'' in the sense of (\ref{eqn-normal-bundles}), by the Tubular Neighbourhood Theorem (see e.g. Theorem 6.5 of \cite{CdS}) it suffices to consider infinitesimal deformations valued in $J(TL)$.

Given a totally real submanifold $f: L^3 \to X$, the splitting
\[
f^* TX = TL \oplus J(TL)
\]
means that for any $p \in L$ we can write any $v \in T_{f(p)}X$ uniquely as $v = u + Jw$ for some $u,w \in T_{f(p)}L$. Lotay and Pacini \cite{LotayPacini} introduce the projections $\pi_L: f^*TX \to TL$ and $\pi_J: f^*TX \to J(TL)$ where $\pi_L(v) = u$ and $\pi_J(v)=Jw$. They also note that
\[
\pi_L(Jv) = \pi_L(Ju - w) = -w = J\pi_J(v),
\]
so we have
\begin{equation}\label{eqn-projs}
\pi_L \circ J = J \circ \pi_J.
\end{equation}

In Section~\ref{subsec-linearizations}, we use the following lemma when showing surjectivity of the operator $\eL$.

\begin{lemma}\label{lem-2-form-ext}
Let $(X,J,\omega)$ be a complex manifold with a Hermitian metric, and let $f: L \hookrightarrow (X,J,\omega)$ be a $C^{k,a}$ embedded totally real submanifold. Then a $C^{k-1,a}$ $2$-form $\eta$ on $L$ can be extended to a $C^{k-1,a}$ real $(1,1)$-form $\sigma$ on $X$. Furthermore for $1 \leq \ell \leq k-1$, the pullback map
\[
f^*: C^{\ell,a}(\Lambda^{1,1}_{\R}(X)) \to C^{\ell,a}(\Lambda^2(L))
\]
is a bounded linear operator between Banach spaces.
\end{lemma}

\begin{proof}
First, we show the result locally. Let $\varepsilon^1, \ldots, \varepsilon^n$ be a $C^{k-1,a}$ local orthonormal frame for $T^*L$ on an open coordinate patch $U \subset L$. We can write a $C^{k-1,a}$ $2$-form $\eta$ on $U$ using the local expression
\[
\eta = \frac{1}{2} \eta_{jk} \varepsilon^j \wedge \varepsilon^k,
\]
where $\eta_{jk} \in C^{k-1,a}(U,\R)$ for $j,k \in \{1,\ldots,n\}$.

By the totally real condition, we see that $\varepsilon^1, \ldots, \varepsilon^n, J\varepsilon^1, \ldots, J\varepsilon^n$ is a local frame for the restricted bundle $T^*X \vert_U$. This frame can be extended to a $C^{k-1,a}$ local frame for $T^*X$ on an open set $\tilde{U} \subset X$ such that $\tilde{U} \cap L = U$. Using this extended frame, we construct $C^{k-1,a}$ local frames $\varphi^1, \ldots, \varphi^n$ and $\overline{\varphi}^1, \ldots, \overline{\varphi}^n$ for $\Lambda^{1,0}(X)$ and $\Lambda^{0,1}(X)$ on $\tilde{U}$, respectively, by defining
\begin{equation}\label{eqn-holc-frame-def}
\varphi^k = \frac{1}{\sqrt{2}}(\varepsilon^k + iJ\varepsilon^k), \hspace{0.5cm} \overline{\varphi}^k = \frac{1}{\sqrt{2}}(\varepsilon^k - iJ\varepsilon^k).
\end{equation}
Extend the functions $\eta_{jk}$ arbitrarily to $\tilde{\eta}_{jk} \in C^{k-1,a}(\tilde{U}, \R)$ and consider the $C^{k-1,a}$ real $(1,1)$-form $\sigma$ on $\tilde{U}$ given by
\begin{equation}\label{eqn-ambient-form}
\sigma = \frac{1}{2} \tilde{\eta}_{jk} \left(\varphi^j \wedge \overline{\varphi}^k + \overline{\varphi}^j \wedge \varphi^k \right)
\end{equation}
Expanding (\ref{eqn-ambient-form}) using (\ref{eqn-holc-frame-def}), we obtain
\[
\sigma = \frac{1}{2} \tilde{\eta}_{jk} \left(\varepsilon^j \wedge \varepsilon^k + (J\varepsilon^j) \wedge (J\varepsilon^k)\right).
\]
Given any $p \in U$ and $v \in T_pU$, we have
\begin{align*}
f^* (J \varepsilon^k)(v) &= \varepsilon^k(J f_* v)\\
&= \varepsilon^k(\pi_L \circ J (f_* v)) \\
&= \varepsilon^k(J \circ \pi_J (f_* v)) = 0,
\end{align*}
for all $k$, where we used (\ref{eqn-projs}) in the last line. Therefore $f^* (J \varepsilon^k) = 0$, and we conclude that $f^* \sigma = \eta$.

Now, to obtain a global solution, consider a finite open cover $\{U_r\}_{r=1}^N$ by coordinate patches of $L$. In each of these open sets, extend $\eta \vert_{U_r}$ as above to a $C^{k-1,a}$ real $(1,1)$-form $\sigma_r$ on an open set $\tilde{U}_r \subset X$ such that $\tilde{U}_r \cap L = U_r$.

Now we use a partition of unity argument. Let $\{\psi_r: X \to [0,1]\}_{r=1}^N$ be a collection of smooth functions into $[0,1]$ such that $\supp \psi_r \ssubset \tilde{U}_r$ for each $r$, and
\[
\sum_{r=1}^N \psi_r(x) = 1 
\]
for all $x \in L$. We then have that
\[
\sigma := \sum\limits_{r=1}^N \psi_r \sigma_r
\]
is a $C^{k-1,a}$ real $(1,1)$-form on $X$, vanishing outside of $\cup_{r=1}^N \tilde{U}_r$, such that $f^* \sigma = \eta$.

Now, to see that $f^*: C^{\ell,a}(\Lambda^{1,1}_{\R}(X)) \to C^{\ell,a}(\Lambda^2(L))$ is bounded when $1 \leq \ell \leq k-1$, we first choose $\epsilon > 0$ smaller than the injectivity radii of both $L$ and $X$. Then for any two points $x,y \in L$, there is a unique length minimizing geodesic from $y$ to $x$, along with parallel transport maps $\Pi^L_{yx}, \Pi^X_{yx}$ which are linear isometries between fibres of $T^*L, T^*X$, as well as fibres of the tensor powers $(T^*L)^{\otimes k}, (T^*X)^{\otimes k}$ and their antisymmetrizations. Since the pullback $f^*$ is pointwise norm non-increasing, we have the inequalities
\begin{align*}
\norm{f^* \sigma}_{C^{\ell,a}(\Lambda^2(L))} &= \sup\limits_{0 < d_L(x,y) < \epsilon} \frac{1}{d_L(x,y)^a} \norm{f^*(\nabla^{\ell} \sigma) (x) - \Pi^L_{yx} f^*(\nabla^{\ell} \sigma)(y)} \\
&\hspace{1cm} + \sup\limits_{x \in L} \norm{(f^*\sigma)(x)} + \sup\limits_{x \in L} \norm{f^*(\nabla^{\ell} \sigma)(x)} \\
& \leq \sup\limits_{0 < d_X(f(x),f(y)) < \epsilon} \frac{1}{d_X(f(x),f(y))^a} \norm{\sigma (f(x)) - \Pi^X_{yx} \nabla^{\ell} \sigma(f(y))}  \\
&\hspace{1cm} + \sup\limits_{x \in L} \norm{\sigma(f(x))} + \sup\limits_{x \in L} \norm{\nabla^{\ell} \sigma(f(x))} \\
& \leq \norm{\sigma}_{C^{\ell,a}(\Lambda^{1,1}_{\R}(X))},
\end{align*}
so that $f^*$ is bounded, as desired. In fact, $f^*$ has norm at most $1$.
\end{proof}

\section{Structure of the PSL moduli space}

\subsection{The parameter space and the moduli space}\label{subsec-moduli-space}


Let $X$ be a non-K\"ahler Calabi--Yau threefold, and let $L^3$ be a fixed compact manifold of real dimension $3$. The space $C^{\ell,a}(L,X)$ of $C^{\ell,a}$ H\"older maps $L \to X$ is a separable Banach manifold.\footnote{Banach manifolds, which arise often in global analysis, are infinite dimensional manifolds locally modelled on open sets in Banach spaces. We refer the interested reader to Moore's book \cite{Moore} for a detailed and readable treatment of the theory.} A Riemannian metric $g$ on $X$ is required to make sense of the H\"older condition, but since $X$ is compact the choice of metric does not affect which maps are elements of $C^{\ell,a}(L,X)$. The tangent space at $f \in C^{\ell,a}(L,X)$ is the space of $C^{\ell,a}$ vector fields along $f$:
\[
T_f C^{\ell,a}(L,X) \cong C^{\ell,a}(f^* TX).
\]

It is well known that the space $\Emb^{\ell,a}(L,X)$ of $C^{\ell,a}$ embeddings $L \to X$ is an open Banach submanifold of $C^{\ell,a}(L,X)$. A theorem of Binz and Fischer \cite{BinzFischer} states that $\Emb^{\ell,a}(L,X)$ is the total space of a principal bundle whose structure group is the group $\Diff(L)$ of diffeomorphisms of $L$, and whose base
\[
B^{\ell,a}(L,X) \cong \Emb^{\ell,a}(L,X) / \Diff(L)
\]
is a manifold consisting of all embedded $C^{\ell,a}$ submanifolds in $X$ of type $L$. The action of $\Diff(L)$ on each fibre is by reparameterization -- i.e. it is the precomposition action $f \overset{\phi}{\mapsto} f \circ \phi$ for $\phi \in \Diff(L)$. Furthermore, the base manifold $B^{\ell,a}(L,X)$ possesses a natural differentiable structure making the quotient map $\Emb^{\ell,a}(L,X) \to B^{\ell,a}(L,X)$ differentiable. If $\ell < \infty$, then $B^{\ell,a}$ is a Banach manifold.

Given the (arbitrary) choice of metric $g$ on $X$, and given $f_0 \in \Emb^{\ell,a}(L,X)$, we can form a chart for $B^{\ell,a}$ and a trivialization of $\Emb^{\ell,a}(L,X) \to B^{\ell,a}$ about the point $[f_0] = f_0 \circ \Diff(L) \in B^{\ell,a}(L,X)$ in the following way: take a tubular neighbourhood $\T$ of $f_0(L)$, with the associated projection map $\pi: \T \to f_0(L)$. The tube $\T$ is the image under a diffeomorphism of some neighbourhood $U$ of the zero-section in the normal bundle $N^{f_0}L \cong (TL)^{\perp} \subset f_0^*TX$. That diffeomorphism $\Phi: U \to \T$ can be taken to be the exponential map in $TX$ restricted to $U$:
\[
\Phi(x,\xi) = \exp_{f_0(x)} \xi, \hspace{0.5cm} (x,\xi) \in U.
\]
Let $\widehat{\mathcal{U}}^{\ell,a} \subset \Emb^{\ell,a}(L,X)$ denote the set
\[
\widehat{\mathcal{U}}^{\ell,a} := \left\{ f \in \Emb^{\ell,a}(L,X) : f(L) \subset \T\right\},
\]
and let
\[
\mathcal{U}^{\ell,a} := \widehat{\mathcal{U}}^{\ell,a} / \Diff(L) = \left\{[f] \in B^{\ell,a} : f \in \widehat{\mathcal{U}}^{\ell,a}\right\}.
\]
The key fact is that if $f \in \widehat{\mathcal{U}}^{\ell,a}$, then there exists a unique representative $f' \in [f]$ such that $f'(x) = \Phi(x,\xi_x)$ for all $x \in L$, where $\xi$ is a $C^{\ell,a}$ section of $N^{f_0}L$ with $\xi_x \in U$ for all $x \in L$. Namely, take $f' = f \circ (\pi \circ f)^{-1}$ and $\xi_x = \exp_{f(x)}^{-1} f \circ (\pi \circ f)^{-1} (x)$. The unique representative $f' \in [f]$ is called a \emph{local canonical representative} in the coordinate patch $\mathcal{U}^{\ell,a}$, and the map $[f] \to f'$ is a \emph{canonical local section} of $\Emb^{\ell,a}(L,X) \to B^{\ell,a}$. The mapping $\mathcal{U}^{\ell,a} \to U$ with $[f] = [f'] \mapsto \xi$ is a chart mapping nearby $C^{\ell,a}$ submanifolds of type $L$ to small $C^{\ell,a}$ sections of $N^{f_0} L$.
%

This discussion also shows that the tangent space $T_{[f_0]} B^{\ell,a}$ at the point $[f_0] \in B^{\ell,a}$ can be identified with $C^{\ell,a}(N^{f_0}L)$, which for $f_0:L \to X$ totally real can be expressed as
\[
T_{[f_0]} B^{\ell,a} = C^{\ell,a}\left((J \circ f_0)_*(TL)\right).
\]

%
Now, we will start to define a space of parameters for our moduli space problem. Define a set $\widehat{\mathcal{B}}^{\ell,a}$ by
\begin{align*}
\widehat{\B}^{\ell,a} := & \bigcup_{\underset{\omega \in \Herm^{\ell-1,a}(X)}{f \in \Emb^{\ell,a}(L,X)}} \left(\{f\} \times \left(\T^{\ell+1,a}_{f,\omega} \cap \mathcal{H}^3(L)^{\perp_{\omega}}\right) \times \{\omega\}\right) \\
& \hspace{3cm} \subset \Emb^{\ell,a}(L,X) \times  C^{\ell+1,a}\left(\Lambda^3(L)\right) \times \Herm^{\ell-1,a}(X).
\end{align*}
Here $\Herm^{\ell-1,a}(X)$ denotes the space of $C^{\ell-1,a}$ Hermitian metrics on $X$, and $\T_{f,\omega}^{\ell+1,a}$ is as defined in (\ref{3-form-tube-def}).
%

Since $\norm{\ed_{f^*\omega}^{\dagger}}$ (see Lemma~\ref{lem-tot-real}) depends continuously on $f$ and $\omega$, it follows that $\T_{f,\omega}^{\ell+1,a}$ is an open subset of $C^{\ell+1,a}(\Lambda^3(L))$ for each $f$ and $\omega$, and hence $\T^{\ell+1,a}_{f,\omega} \cap \mathcal{H}^3(L)^{\perp_{\omega}}$ is an open subset of $\mathcal{H}^3(L)^{\perp_{\omega}}$. Thus $\widehat{\B}^{\ell,a}$ is a Banach submanifold of codimension $b_3(L)$ in the product Banach manifold $C^{\ell,a}(L,X) \times C^{\ell+1,a}\left(\Lambda^3(L)\right) \times \Herm^{\ell-1,a}(X)$. We define $\widehat{\B}^{\ell,a}$ in this way so that if $(f,\rho,\omega) \in \widehat{\mathcal{B}}^{\ell,a}$ and $f^* \omega + \ed^{\dagger}_{f^*\omega} \rho=0$, then $f$ is totally real, by Lemma~\ref{lem-tot-real}, and the reason that we take the second factor to be the intersection $\T^{\ell+1,a}_{f,\omega} \cap \mathcal{H}^3(L)^{\perp_{\omega}}$ is that we do not wish to distinguish between $\rho$ and $\rho'$ when $\ed^{\dagger}_{f^* \omega} \rho = \ed^{\dagger}_{f^*\omega} \rho'$.
%

Since the set $\T_{f,\omega}$ is defined in terms of the norm $\norm{\ed^{\dagger}_{f^* \omega}}$, and we have
\[
\norm{\ed^{\dagger}_{f^* \omega}}_{f^*{\omega}} = \norm{\ed^{\dagger}_{(f \circ \phi)^* \omega}}_{(f \circ \phi)^* \omega}
\]
for all $\phi \in \Diff(L)$, we observe that $\T_{(f\circ \phi),\omega} = \T_{f,\omega}$. I.e. the set depends only on the $\Diff(L)$-orbit of $f$, and we can denote the set by $\T_{[f],\omega}$. Therefore we can consider a well-defined action of $\Diff(L)$ of $L$ on $\widehat{\B}^{\ell,a}$ by
\[
(f, \rho, \omega) \overset{\phi}{\mapsto} (f \circ \phi, \phi^*\rho, \omega).
\]


We define the \textbf{parameter space} $\B^{\ell,a}$ to be the quotient of $\widehat{\B}^{\ell,a}$ by this group action:
\begin{align*}
\B^{\ell,a} &:= \widehat{\B}^{\ell,a} / \Diff(L) \\
&\cong \bigcup_{\underset{\omega \in \Herm^{\ell-1,a}(X)}{[f] \in B^{\ell,a}(L,X)}} \left(\{[f]\} \times \left(\T^{\ell+1,a}_{[f],\omega} \cap \mathcal{H}^3(L)^{\perp_{\omega}}\right) \times \{\omega\}\right).
\end{align*}

Let $\widehat{\M}^{\ell,a} \subset \B^{\ell,a}$ denote the set of all triples $(f,\rho,\omega) \in \widehat{\B}^{\ell,a}$ such that $f: (L,\rho) \to (X,J,\omega,\Omega)$ is PSL. Observe that if $(f,\rho,\omega) \in \widehat{\M}^{\ell,a}$ and $\phi \in \Diff(L)$, then $(f \circ \phi, \phi^* \rho, \omega) \in \widehat{\M}^{\ell,a}$ as well. This follows by a calculation
\begin{align*}
&(f \circ \phi)^* \omega = \phi^* f^* \omega = -\phi^* \ed^{\dagger}_{f^* \omega} \rho = -\ed^{\dagger}_{(f \circ \phi)^* \omega} \phi^* \rho,   \\
& (f \circ \phi)^* \Im \Omega = \phi^* f^* \Im \Omega = 0.
\end{align*}

Thus, $\widehat{\M}^{\ell,a}$ is $\Diff(L)$-saturated in $\widehat{\B}^{\ell,a}$, and it makes sense to talk about its image $\M^{\ell,a}$ in the quotient $\B^{\ell,a}$. We refer to $\M^{\ell,a}$ as the \textbf{universal moduli space}, in analogy with the terminology for pseudoholomorphic curves (see \cite{McDuffSalamon}). On the other hand, for $\omega \in \Herm^{\ell-1,a}(X)$, we will let $\M_{\omega}$ denote the moduli space of $C^{\ell,a}$ PSL submanifolds $f: (L,\rho) \to (X,J,\omega,\Omega)$. I.e. $\M_{\omega}$ is the set of $[f,\rho]$ such that $\left([f,\rho],\omega\right) \in \M^{\ell,a}$.
%

%

\subsection{Reduction to the local picture}\label{subsec-linearizations}

It will be convenient to write down a local chart for the parameter space $\B^{\ell,a}$. Let $([f_0,\rho_0],\omega_0) \in \B^{\ell,a}$. In the last section, we described how to construct a chart $\U^{\ell,a} \subset B^{\ell,a}$ about $[f_0]$ by identifying nearby embedded submanifolds in $X$ with small sections of the normal bundle along $f_0$. Define a set $\mathcal{V}^{\ell,a} \subset \B^{\ell,a}$ by
\begin{equation}\label{def-chart-holder}
\mathcal{V}^{\ell,a} := \bigcup_{\underset{\omega \in \Herm^{\ell-1,a}(X)}{[f] \in \mathcal{U}^{\ell,a}}} \left(\{[f]\} \times \left(\T^{\ell+1,a}_{[f],\omega} \cap \mathcal{H}^3(L)^{\perp}\right) \times \{\omega\}\right).
\end{equation}
The open set $\mathcal{V}^{\ell,a} \subset \B^{\ell,a}$ is a chart in $\B^{\ell,a}$ whose coordinate map is the coordinate map previously defined on $\mathcal{U}^{\ell,a}$ times the projection maps onto the $C^{\ell,a}(\mathcal{H}^3(L)^{\perp})$ and $\Herm^{\ell-1,a}(X)$ factors.

Consider a point $P_0 = ([f_0,\rho_0],\omega_0) \in \M$ (i.e. we have a PSL submanifold $f_0: (L, \rho_0) \to (X,J,\omega_0,\Omega)$), and the chart $\mathcal{V}^{\ell,a}$ about $P_0$ as described above. For all $\alpha$ in a small neighbourhood $\T_1$ of the zero section in $C^{\ell,a}(\Lambda^1(L))$, there is a well-defined deformation $f_{\alpha} : L \to X$ given by
\begin{equation}\label{eqn-def-def}
f_{\alpha}(x) = \exp_{f_0(x)} J f_* \alpha^{\sharp_0}(x),
\end{equation}
where $\exp$ and $\sharp_0$ are taken with respect to $\omega_0$. The mapping $f_{\alpha} \mapsto \alpha$ is none other than the chart on $\mathcal{U}^{\ell,a}$ described above, and the mapping $[f_{\alpha}] \to f_{\alpha}$ is the canonical local section in the chart $\mathcal{U}^{\ell,a}$.

Now, consider the Banach vector bundle $\Y \to \mathcal{V}^{\ell,a}$ whose fibre over a point $(\alpha,\rho,\omega) \in \mathcal{V}^{\ell,a}$ is given by
\[
\Y_{(\alpha,\rho,\omega)} := C^{\ell-1,a}\left(\Lambda^2(L)\right) \oplus \ed_{f_{\alpha}^*\omega}^{\dagger}C^{\ell,a}\left(\Lambda^1(L)\right).
\]
For convenience, let $\Y$ denote the fibre over $(f_0,\rho_0,\omega_0)$:
\begin{equation}\label{Y-def}
\Y :=C^{\ell-1,a}\left(\Lambda^2(L)\right) \oplus \ed_0^{\dagger}C^{\ell,a}\left(\Lambda^1(L)\right).
\end{equation}
Here and elsewhere $\ed_0^{\dagger} := \ed^{\dagger}_{f_0^*\omega_0}$. Let $s = s_{\mathcal{V}}^{\ell,a}: \mathcal{V}^{\ell,a} \to \Y$ be the section of $\Y$ given by
\begin{align}
\begin{split}\label{eqn-sec-def}
&\news(\alpha,\rho,\omega)= (f_{\alpha}^* \omega + \ed^{\dagger}_{\omega} \rho) + \star_{\omega} f_{\alpha}^* \Im \Omega \\
&\hphantom{\news(\alpha,\rho,\omega)} =: \newstwo + \newszero.
\end{split}
\end{align}

In order to see that $s$ is valued in $\Y$, the only thing to check is that the function $\star_{\omega} f_{\alpha}^* \Im \Omega$ is a co-exact $0$-form on $L$. This follows by a standard homotopy argument. Since $f_0^* \Im \Omega = 0$ and $\ed\left(\Im \Omega\right)=0$, we see that
\begin{align*}
f_{\alpha}^* \Im \Omega &= \int_0^1 \frac{\ed}{\ed t} (f_{t \alpha}^* \Im \Omega)\,\ed t = \int_0^1 \mathscr{L}_{J ((f_{t\alpha})_*\alpha)^{\sharp_0}} \Im \Omega\,\ed t  \\
&= \int_0^1 \ed \left(\iota_{J ((f_{t\alpha})_*\alpha)^{\sharp}} \Im \Omega \right) \,\ed t = \ed \hspace{-0.1cm}\int_0^1 \left(\iota_{J ((f_{t\alpha})_*\alpha)^{\sharp}} \Im \Omega \right)\,\ed t,
\end{align*}
so that $f_{\alpha}^* \Im \Omega$ is exact, and hence $\star_{\omega} f_{\alpha}^* \Im \Omega$ is co-exact.

The \textbf{local universal moduli space} is the set $\M(\mathcal{V}^{\ell,a}) = \M^{\ell,a} \cap \mathcal{V}^{\ell,a}$. The reason we defined the bundle $\Y^{\ell,a}$ and the section $s$ above is so that we can write the local universal moduli space as the set of zeroes of $s$:
\[
\M(\mathcal{V}^{\ell,a}) = s^{-1}(0).
\]
Given $\omega \in \Herm^{\ell-1,a}(X)$, we define $\M_{\omega}(\mathcal{V}^{\ell,a})$ to be the set of $(f_{\alpha},\rho)$ such that $(f_{\alpha},\rho,\omega) \in \M(\mathcal{V}^{\ell,a})$. In Section~\ref{subsec-transvers}, we will prove the following local version of the main theorem:
\begin{theorem}[Local version of Theorem~\ref{thm-main}, $C^{\ell,a}$ parameters]\label{thm-main-loc-holder}
There is a comeagre set of Hermitian metrics $\Herm^{\ell-1,a}_{\reg,\mathcal{V}^{\ell,a}}(X) \subset \Herm^{\ell-1,a}(X)$ such that for all $\omega \in \Herm^{\ell-1,a}_{\reg,\mathcal{V}^{\ell,a}}(X)$, the set $\M_{\omega}(\mathcal{V}^{\ell,a})$ consists of isolated points.
\end{theorem}



In order to understand the structure of the universal moduli space, we would first like to understand the deformation theory of the PSL system. To that end, we consider the linearization
\[
\eL := (\mathrm{D}s)_{P_0},
\]
which is a linear map from the tangent space $T_{P_0} \mathcal{V}^{\ell,a}$ to $\Y_{P_0}$ (\ref{Y-def}). Since $f_0: L \to X$ is totally real, there is an isomorphism from the cotangent bundle $\Lambda^1(L)$ to the normal bundle $J f_*(TL)$ given by $\alpha \mapsto Jf_* \alpha^{\sharp_0}$, where the sharp operator $\sharp_0$ is defined with respect to $\omega_0$. We can then conclude that, the tangent space $T_{P_0}\mathcal{V}^{\ell,a}$ is given by
\begin{equation}\label{X-def}
\X := C^{\ell,a}(\Lambda^1(L)) \oplus C^{\ell+1,a}\left(\mathcal{H}^3(L)^{\perp_{\omega_0}} \right) \oplus C^{\ell-1,a}(\Lambda^{1,1}_{\R}(X)).
\end{equation}

To avoid confusion, elements of one of the tangent spaces will be denoted by symbols that are decorated with a dot. E.g. $\dot\omega$ will denote an element of $C^{\ell,a}(\Lambda^{1,1}_{\R}(X)) \cong T_{\omega_0} \Herm^{\ell,a}(X)$.

Thus $\eL$ is a map
\begin{equation}
\begin{split}\label{L-dom-codom}
&\eL : \underset{\X}{\underbrace{C^{\ell,a}(\Lambda^1(L)) \oplus C^{\ell+1,a}\left(\mathcal{H}^3(L)^{\perp_{\omega_0}} \right) \oplus C^{\ell-1,a}(\Lambda^{1,1}_{\R}(X))}}\\
&\hspace{2cm} \to \underset{\Y}{\underbrace{C^{\ell-1,a}\left(\Lambda^2(L)\right) \oplus \ed_0^{\dagger}C^{\ell,a}\left(\Lambda^1(L)\right)}}.
\end{split}
\end{equation}
%

Since there are natural splittings of the domain (\ref{X-def}) and codomain (\ref{Y-def}) of $\eL$ into three subspaces and two subspaces, respectively, it will be convenient to write $\eL$ as a $2 \times 3$ formal Jacobian matrix
\[
\eL= \begin{pmatrix}
\dfrac{\partial \newstwo}{\partial \alpha}(P_0) & \dfrac{\partial \newstwo}{\partial \rho}(P_0) & \dfrac{\partial \newstwo}{\partial \omega}(P_0) \\
 & \\
\dfrac{\partial \newszero}{\partial \alpha}(P_0) & \dfrac{\partial \newszero}{\partial \rho}(P_0) & \dfrac{\partial \newszero}{\partial \omega}(P_0)
\end{pmatrix}.
\]

We compute $\eL$ in the following lemma (similar to Lemma~2.2 in \cite{CGPY23}):
\begin{lemma}\label{lem-compute-L}
The formal Jacobian matrix is given by
\[
\eL= \begin{pmatrix}
-\ed +T_0 & -\ed^{\dagger}_0 & f_0^* \\
 & & & \\
\ed^{\dagger}_0 \circ \mathrm{mult}_{|\Omega|_{\omega_0}} & 0 & 0
\end{pmatrix}
+ \begin{pmatrix}
0 & 0 & -\left.\frac{\partial}{\partial \omega} \right|_{\omega_0} \left(\ed^{\dagger}_{f_0^*\omega} \rho_0\right) \\
& & \\
0 & 0 & 0
\end{pmatrix},
\]
where the map $T_0: \Lambda^1(L) \to \Lambda^2(L)$ is defined by 
\[
T_0\alpha = f_0^*(\iota_{J\alpha^{\sharp}} \ed \omega_0),
\]
and the action of $\eL$ on a formal column vector $\begin{pmatrix} \alpha \\ \dot{\rho} \\ \dot{\omega} \end{pmatrix}$ is given by
\[
\eL
\begin{pmatrix}
\alpha \\
\dot{\rho} \\
\dot{\omega}
\end{pmatrix} = 
\begin{pmatrix}
-\ed\alpha  + T_0\alpha -\ed^{\dagger}_0 \dot{\rho} + f_0^* \dot{\omega} -\left.\frac{\partial}{\partial \omega} \right|_{\omega_0} \left(\ed^{\dagger}_{f_0^*\omega} \rho_0\right) \dot\omega \\
\\
\ed^{\dagger}_0 (|\Omega|_{\omega_0} \alpha)
\end{pmatrix}.
\]
\end{lemma}

\begin{proof}
We will first show the linearization
\[
\dfrac{\partial f_{\alpha}^* \omega}{\partial \alpha}(P_0) = -\ed\alpha  + T_0\alpha.
\]
We have:
\begin{align*}
\dfrac{\partial f_{\alpha}^* \omega}{\partial \alpha}(P_0) &= \left. \frac{\ed}{\ed t} \right\vert_{t=0} f_{t\alpha}^* \omega_0 = f_0^* (\mathscr{L}_{(J \circ \sharp_0) \alpha}  \omega_0) \\
&= f_0^* (\iota_{(J \circ \sharp_0) \alpha} \ed \omega_0 + \ed(\iota_{(J \circ \sharp_0) \alpha} \omega_0)) \\
&= T_0\alpha - \ed\alpha ,
\end{align*}
where the last equality follows from the definition of $T_0$ and the computation
\begin{align*}
\iota_{J \alpha^{\sharp_0}} \omega_0 = -\omega_0(\cdot, J\alpha^{\sharp_0}) = -g_0(\cdot,\alpha^{\sharp_0}) = -\alpha.
\end{align*}
Now in order to show that
\[
\dfrac{\partial f_{\alpha}^* \Im \Omega}{\partial \alpha}(P_0) = \ed^{\dagger}_0(|\Omega|_{\omega_0} \alpha),
\]
let $x^1,x^2,x^3,y^1,y^2,y^3$ be local real coordinates for a coordinate patch $U$ in $X$ intersecting $f_0(L)$. Since $f_0: L \to X$ is totally real, we can assume that $U \cap f_0(L) = \{y^1=y^2=y^3=0\}$ and that $J\frac{\partial}{\partial x^i} = \frac{\partial}{\partial y^i}$ along $U \cap f_0(L)$ for $i=1,2,3$. Furthermore we have
\begin{align*}
\Omega &= |\Omega|_{\omega_0} \ed z^1 \wedge \ed z^2 \wedge \ed z^3 \\
&=|\Omega|_{\omega_0} (\ed x^1 +i\ed y^1) \wedge (\ed x^2 +i\ed y^2) \wedge (\ed x^3 +i\ed y^3).
\end{align*}

It suffices to prove the linearization formula when $\alpha = \ed x^i$, which we can then extend by linearity. Thus
\begin{align}
\begin{split}\label{eqn-ImOm-coord}
&\Im \Omega = |\Omega|_{\omega_0} \left(\ed y^1 \wedge \ed x^2 \wedge \ed x^3 + \ed x^1 \wedge \ed y^2 \wedge \ed x^3\right. \\
&\hspace{2.5cm}  \left.+ \ed x^1 \wedge \ed x^2 \wedge \ed y^3 -\ed y^1 \wedge \ed y^2 \wedge \ed y^3\right).
\end{split}
\end{align}
Since $\Im \Omega$ is a closed $3$-form, we have
\begin{align*}
\left. \frac{\ed}{\ed t} \right\vert_{t=0} f_{t\ed x^i}^* \Im \Omega &= f_0^*\ed(\iota_{\frac{\partial}{\partial y^i}} \Im \Omega).
\end{align*}
Substituting the expression (\ref{eqn-ImOm-coord}), we see that
\[
\left. \frac{\ed}{\ed t} \right\vert_{t=0} f_{t\ed x^i}^* \Im \Omega = \ed \star_0 \left(|\Omega|_{\omega_0} \ed x^i\right),
\]
so that
\[
\left. \frac{\ed}{\ed t} \right\vert_{t=0} \star_0 f_{t\ed x^i}^* \Im \Omega = \ed^{\dagger}_0 \left(|\Omega|_{\omega_0} \ed x^i\right).
\]
The remaining entries of $\eL$ are evident, since the corresponding terms in $\news$ are linear.
\end{proof}
%

\subsection{Infinitesimal PSL deformations}\label{subsec-inf-def}
In addition to $\eL$, we also consider an important pair of submatrices of $\eL$:
\[
\eL_1 := \begin{pmatrix}
-\ed +T_0 & -\ed^{\dagger}_0 \\
 & & \\
\ed^{\dagger}_0 \circ \mathrm{mult}_{|\Omega|_{\omega_0}} & 0 
\end{pmatrix}
\hspace{0.75cm}
\eL_2 := \begin{pmatrix}
f_0^* + \left.\frac{\partial}{\partial \omega} \right|_{\omega_0} \left(\ed^{\dagger}_{f_0^*\omega} \rho_0\right) \\
\\
0
\end{pmatrix}.
\]
We denote the domains of $\eL_1$ and $\eL_2$ by $\X_1$ and $\X_2$ respectively, so that $\X = \X_1 \oplus \X_2$ and $\eL = \eL_1 \oplus \eL_2$, where
\begin{align*}
\X_1 &= C^{\ell,a}(\Lambda^1(L)) \oplus C^{\ell+1,a}\left(\mathcal{H}^3(L)^{\perp}\right)\\
\X_2 &= C^{\ell-1,a}\left(\Lambda^{1,1}_{\R}(X)\right).
\end{align*}

The operators $\eL_1$ and $\eL_2$ can be understood as partial derivatives of $\news$ along the subspaces $\X_1$ and $\X_2$ of $\X$, respectively. The importance of $\eL_1$ is that given the fixed Hermitian metric $\omega_0$, the kernel of $\eL_1$ corresponds via the isomorphism $J\circ \sharp_0$ to infinitesimal deformations of $L$ valued in $J(TL)$ which preserve the PSL condition. These infinitesimal deformations are exactly the solutions of the system
\begin{align}\label{eqns-inf-def}
\begin{split}
\ed\alpha  + \ed^{\dagger}_0 \dot{\rho} - T(\omega_0)\alpha &= 0 \\
\ed^{\dagger}_0\left(|\Omega|_{\omega_0}\alpha\right) &= 0.
\end{split}
\end{align}
Notice by applying $\ed$ to the first equation, we see that any infinitesimal PSL deformation also has to satisfy the equation
\begin{align*}
&\Delta \dot{\rho} = \mathscr{L}_{J\alpha^{\sharp}} \ed\omega.
\end{align*}

One consequence of the following lemma is that the space of solutions of (\ref{eqns-inf-def}) has a finite dimension.

\begin{lemma}\label{lem-ellip}
Recall the Banach space $\Y = \Y_{P_0}$ given by
\[
\Y := C^{\ell-1,a}\left(\Lambda^2(L)\right) \oplus \ed_0^{\dagger}C^{\ell,a}\left(\Lambda^1(L)\right).
\]
The operator 
\begin{align*}
&\eL_1 : \X_1 \to \Y \\
&\eL_1 \begin{pmatrix}
\alpha \\ \\ \dot\rho
\end{pmatrix}
=  \begin{pmatrix}
-\ed +T(\omega_0) & -\ed^{\dagger}_0 \\
 & & \\
\ed^{\dagger}_0 \circ \mathrm{mult}_{|\Omega|_{\omega_0}} & 0 
\end{pmatrix}
\begin{pmatrix}
\alpha \\ \\ \dot\rho
\end{pmatrix}
= \begin{pmatrix}
-\ed\alpha  + T(\omega_0)\alpha - \ed_0^{\dagger} \dot\rho \\
\\
\ed_0^{\dagger}\left(|\Omega|_{\omega_0} \alpha\right)
\end{pmatrix}
\end{align*} is elliptic. Its Fredholm index is $\mathrm{ind}(\eL_1) = 0$.
\end{lemma}

\begin{proof}
Consider the path of operators $t \mapsto \eL_1^t$, where for each $t \in [0,1]$ we have
\[
\eL_1^t(\alpha + \dot{\rho}) := -\ed\alpha + tT(\omega_0)\alpha -\ed_0^{\dagger} \dot{\rho} +\ed_0^{\dagger}\left(|\Omega|_{\omega_0}^t \alpha\right).
\]

Given $(x,\xi)$ such that $x \in L$, and $\xi \in T_x^*L$ with $\xi \neq 0$, the \textbf{principal symbol} $\sigma_{(x,\xi)}(\eL_1^t)$ is given by the following limit, where $\phi \in C^{\infty}(L)$ is any function satisfying $\ed\phi(x) = \xi$:
\begin{align*}
\sigma_{(x,\xi)}(\eL_1^t) (\alpha + \dot{\rho}) &= \lim\limits_{s \to \infty} s^{-1} e^{-s\phi(x)} \eL_1^t\left(e^{s\phi}(\alpha + \dot{\rho})\right).
\end{align*}
Thus to compute the symbol, we compute $\eL_1^t\left(e^{s\phi}(\alpha + \dot{\rho})\right)$ and ignore any terms that are not linear in $s$. We have
\begin{align*}
\eL_1^t\left(e^{s\phi}(\alpha + \dot{\rho})\right) &= -\ed\left(e^{s\phi} \alpha\right) + tT\left(e^{s\phi}\alpha\right) -\ed_0^{\dagger}\left(e^{s\phi} \dot{\rho}\right) + \ed_0^{\dagger} \left(|\Omega|_{\omega_0}^t e^{s\phi} \alpha\right) \\
&= -\ed(e^{s\phi}) \wedge \alpha - e^{s\phi} \ed\alpha  + e^{s\phi} tT\alpha - e^{s\phi} \ed_0^{\dagger} \dot{\rho} + \iota_{\nabla e^{s\phi}} \dot{\rho}  \\
& \hspace{2cm} +|\Omega|_{\omega_0}^t e^{s\phi} \ed_0^{\dagger} \alpha- \left\langle \ed\left(|\Omega|_{\omega_0}^t e^{s\phi}\right), \alpha \right\rangle \\
&= e^{s\phi} \left[ s\left(-\xi \wedge \alpha - |\Omega|_{\omega_0}^t \langle \xi, \alpha \rangle + \iota_{\xi^{\sharp_0}} \dot{\rho}\right) + \lot \right],
\end{align*}
by standard Leibniz rules for $\ed$ and $\ed_0^{\dagger}$. Thus, the symbol is given by
\begin{align*}
\sigma_{(x,\xi)}(\eL_1^t) (\alpha + \dot{\rho})= -\xi \wedge \alpha + \iota_{\xi^{\sharp_0}} \dot{\rho}- |\Omega|_{\omega_0}^t \langle \xi, \alpha \rangle.
\end{align*}

We will show that for each $x \in L$, whenever $\xi \in T_x^*L$ with $\xi \neq 0$, the symbol is a linear isomorphism. First of all, observe that the dimensions of the domain and codomain are equal and finite, so it suffices to show that $\sigma_{(x,\xi)}(\eL_1^t)$ is injective.

We claim that if $\xi \wedge \alpha - \iota_{\xi^{\sharp_0}} \dot{\rho} = 0$ and $\langle \alpha, \xi \rangle = 0$, then $\dot{\rho} = 0$ and $\alpha = 0$. To that end, let $\{e_1, e_2, e_3\}$ be an orthonormal basis for $T_x L$ and $\{\varepsilon^1, \varepsilon^2, \varepsilon^3 \}$ the dual basis for $T_x^* L$. Then given vectors $v,w \in T_x L$, we have
\[
(\xi \wedge \alpha)(v,w) = \xi(v) \alpha(w) - \xi(w) \alpha(v),
\]
which when expressed in coordinates, with $\xi = \xi_k \varepsilon^k, \alpha = \alpha_k \varepsilon^k$ and $v = v^k e_k, w = w^k e_k$, is given by
\begin{align}
\begin{split}\label{eqn-append1}
(\xi \wedge \alpha)(v,w) &= \left(\xi_1 \alpha_2 - \xi_2 \alpha_1\right) v^1 w^2  + \left(\xi_2 \alpha_1 - \xi_1 \alpha_2\right) v^2 w^1\\
& \hspace{1cm} + \left(\xi_1 \alpha_3 - \xi_3 \alpha_1\right) v^1 w^3  + \left(\xi_3 \alpha_1 - \xi_1 \alpha_3\right) v^3 w^1 \\
& \hspace{2cm}  + \left(\xi_2 \alpha_3 - \xi_3 \alpha_2\right) v^2 w^3 + \left(\xi_3 \alpha_2 - \xi_2 \alpha_3\right) v^3 w^2.
\end{split}
\end{align}
On the other hand, we have
\[
\iota_{\xi^{\sharp_0}} \dot{\rho}(v,w) = |\dot{\rho}| \cdot \det \begin{pmatrix}
\xi_1 & v^1 & w^1 \\
\xi_2 & v^2 & w^2 \\
\xi_3 & v^3 & w^3
\end{pmatrix},
\]
so that
\begin{equation}\label{eqn-append2}
\iota_{\xi^{\sharp_0}} \dot{\rho}(v,w) = |\dot{\rho}| \left(\xi_1 v^2 w^3 - \xi_1 v^3 w^2 + \xi_2 v^3 w^1 - \xi_2 v^1 w^3 + \xi_3 v^1 w^2 - \xi_3 v^2 w^1\right).
\end{equation}
If $\xi \wedge \alpha = \iota_{\xi^{\sharp_0}} \dot{\rho}$, then comparing coefficients in (\ref{eqn-append1}) and (\ref{eqn-append2}) yields the system of equations
\begin{align}
\begin{split}\label{syst-append}
|\dot{\rho}| \xi_1 &= \xi_2 \alpha_3 - \xi_3 \alpha_2 \\
|\dot{\rho}| \xi_2 &= \xi_3 \alpha_1 - \xi_1 \alpha_3 \\
|\dot{\rho}| \xi_3 &= \xi_1 \alpha_2 - \xi_2 \alpha_1.
\end{split}
\end{align}
Identifying $T_x^*L$ with $\R^3$ in the usual way, we interpret (\ref{syst-append}) as saying
\[
|\dot{\rho}| \xi = \xi \times \alpha,
\]
which for $\xi \neq 0$ can only be satisfied if $|\dot{\rho}| = 0$ and $\alpha \parallel \xi$. Then imposing the assumption $\langle \xi, \alpha \rangle = 0$ gives $\alpha = 0$, completing the proof that the symbol $\sigma_{(x,\xi)}(\eL_1^t)$ is a linear isomorphism. We conclude that $\eL_1^t$ is elliptic for each $t \in [0,1]$.


Now, the operator $\eL_1^0$ is given by
\[
\eL_1^0(\alpha + \dot{\rho}) := -\ed\alpha  - \ed_0^{\dagger} \dot{\rho} + \ed_0^{\dagger} \alpha.
\]

We claim that $\ker \eL_1^0 \cong \Hcal^1(L)$. The inclusion $\Hcal^1(L) \subset \ker \eL_1^0$ is clear. For the reverse inclusion, let $\alpha + \dot{\rho} \in \ker \eL_1^0$. Then we have
\begin{align*}
&\ed\alpha + \ed_0^{\dagger} \dot{\rho} = 0, \\
&\ed_0^{\dagger} \alpha = 0.
\end{align*}
By the Hodge decomposition of $\Lambda^2(L)$, the first of these equations implies that $\ed\alpha  = \ed_0^{\dagger} \dot{\rho} = 0$. I.e. $\dot{\rho} \in \Hcal^3(L) \cap \Hcal^3(L)^{\perp}$, so $\dot\rho = 0$. The equations then reduce to $\ed\alpha  = 0$ and $\ed_0^{\dagger}\alpha = 0$, and we get that $\alpha \in \Hcal^1(L)$. It follows that $\dim \ker \eL_1^0 = b_1(L)$.

We claim that $\coker \eL_1^0 \cong \Hcal^2(L)$. In order to prove this, it suffices to show that 
\begin{equation}\label{inc-coker}
\ed^{\dagger}_0 C^{\ell,a}(\Lambda^1(L)) \oplus C^{\ell-1,a}(\Hcal^2(L)^{\perp}) = \im \eL_1^0.
\end{equation}
The inclusion $\im \eL_1^0 \subset \ed^{\dagger}_0 C^{\ell,a}(\Lambda^1(L)) \oplus C^{\ell-1,a}(\Hcal^2(L)^{\perp})$ is clear. To see the reverse inclusion, choose arbitrary elements of $\ed^{\dagger}_0 C^{\ell,a}(\Lambda^1(L))$ and $ C^{\ell-1,a}(\Hcal^2(L)^{\perp})$, which we may write as $\ed^{\dagger} \tau$ and $\ed\eta + \ed_0^{\dagger} \sigma$ respectively, where $\tau \in C^{\ell,a}\left(\Lambda^1(L)\right)$ and $\sigma \in C^{\ell,a}\left(\Hcal^3(L)^{\perp}\right)$

One may find $\alpha \in C^{\ell,a}\left(\Lambda^1(L)\right)$ and $\dot{\rho} \in C^{\ell,a}\left(\Hcal^3(L)^{\perp} \right)$ which satisfy the equations
\begin{align*}
&\ed_0^{\dagger}\tau = \ed_0^{\dagger} \alpha \\
&\ed\eta + \ed_0^{\dagger} \sigma = -\ed\alpha  - \ed_0^{\dagger} \dot{\rho}.
\end{align*}
Indeed, write $\alpha = -\eta + \ed\zeta$ and $\dot{\rho} = -\sigma$, where $\zeta$ solves the equation
\begin{equation}\label{eqn-hmc-0-form}
\Delta \zeta = \ed_0^{\dagger}(\tau + \eta).
\end{equation}
A solution of (\ref{eqn-hmc-0-form}) exists since $\ed_0^{\dagger}(\tau + \eta) \in \Hcal^0(L)^{\perp}$. Then we see that
\[
-\ed\alpha  - \ed_0^{\dagger} \dot{\rho} = \ed\eta + \ed_0^{\dagger}\sigma
\]
and
\begin{align*}
\ed_0^{\dagger}(\tau - \alpha) &= \ed_0^{\dagger}(\tau + \eta - \ed\zeta) \\
&= \ed_0^{\dagger}(\tau + \eta) - \Delta \zeta =0,
\end{align*}
as desired. Thus, we have shown (\ref{inc-coker}). It follows that $\dim \coker \eL_1^0 = b_2(L)$. 

By Poincar\'e duality, we have
\[
\ind(\eL_1^0) = b_1(L) - b_2(L) = 0.
\]
Finally, by the invariance of Fredholm index under deformations, we have $\ind(\eL_1) = 0$ as well, completing the proof.
\end{proof}

While the lemma above shows that $\eL_1$ is Fredholm, note that $\eL_2$ is certainly not Fredholm. Given a real $(1,1)$-form $\dot{\omega}$ on $X$, the map $\eL_2$ does not see any behaviour of $\dot{\omega}$ away from $L$, so $\eL_2$ has an infinite-dimensional kernel. Nevertheless, as we will see in the next section, ellipticity of $\eL_1$ and surjectivity of $\eL = \eL_1 \oplus \eL_2$ are strong enough conditions that we can still prove $s^{-1}(0)$ has a fairly nice structure, namely that of a Banach manifold.

\subsection{The local universal moduli space is a Banach manifold}\label{sec-univ-mod}

Let $F: \M \to \N$ be a $C^1$ map between Banach manifolds. We say that $F$ is a \textbf{submersion} at a point $p \in \M$ if the derivative
\[
\mathrm{D}F_p : T_p \M \to T_y \N
\]
is surjective, and $\ker \mathrm{D}F_p$ splits $T_p \M$, meaning that there exists a closed linear subspace $A \subset T_p \M$ such that $T_p \M = \ker \mathrm{D}F_p \oplus A$. A point $y \in \N$ is said to be a \textbf{regular value} of $F$ if $F$ is a submersion at every $p \in F^{-1}(y)$.

Using the implicit function theorem for Banach spaces, one can prove the following:

\begin{theorem}[Regular value theorem]\label{thm-subm}
Suppose that $F: \M \to \N$ is a $C^1$ map between Banach manifolds, and $y \in \N$ is a regular value of $F$. Then $F^{-1}(y)$ is a Banach submanifold of $\M$.
\end{theorem}

Recall that the local universal moduli space $\M(\mathcal{V}^{\ell,a})$ defined in Section~\ref{subsec-linearizations} is given by the set $s^{-1}(0)$, where $s$ is as defined in (\ref{eqn-sec-def}). We will show that $\M(\mathcal{V}^{\ell,a})$ is a Banach manifold by showing that $0$ is a regular value of $s$ -- i.e. that $\eL = \mathrm{D}s_{P_0}$ is surjective at each $P_0 \in \M(\mathcal{V}^{\ell,a})$ and $\ker \eL$ splits $\X = T_{P_0} \mathcal{V}^{\ell,a}$. Being a submersion is an open condition, so it suffices to show that $\eL$ is a submersion when $\ed^{\dagger}_0 \rho_0 = 0$. Then $\eL$ would still be a submersion when $\rho_0$ is in the small tube $\T^{\ell+1,a}_{f_0,\omega_0} \cap \mathcal{H}^3(L)^{\perp_{\omega_0}}$, since the correction term in Lemma~\ref{lem-compute-L} has a small operator norm. We first show that $\eL$ is surjective when $\ed^{\dagger}_0 \rho_0 = 0$. 

\begin{lemma}\label{lem-transversality}
When $\ed^{\dagger}_0 \rho_0 = 0$, the operator $\eL: \X \to \Y$ is bounded and surjective onto $\Y$.
\end{lemma}

\begin{proof}
Consider an arbitrary element of $\Y$, which we may write as $\eta + \ed_0^{\dagger} \tau$, for some $\eta \in C^{\ell-1,a}(\Lambda^2(L))$ and $\tau \in C^{\ell,a}(\Lambda^1(L))$. Recalling Lemma~\ref{lem-compute-L}, we have
\[
\eL\begin{pmatrix}
\alpha \\
0 \\
\dot{\omega}
\end{pmatrix} =
\begin{pmatrix}
-\ed\alpha  + T_0\alpha + f^* \dot{\omega} \\
\\
\ed^{\dagger}_0 (|\Omega|_{\omega_0} \alpha)
\end{pmatrix} = \begin{pmatrix}
\eta \\
\\
\ed^{\dagger}_0 \tau
\end{pmatrix},
\]
where we have chosen $\alpha$ and $\dot{\omega}$ such that
\begin{align*}
&\alpha = |\Omega|_{\omega_0}^{-1} \tau \\
&f^* \dot{\omega} = \eta + \ed\alpha  - T_0\alpha.
\end{align*}
Lemma~\ref{lem-2-form-ext} guarantees the existence of an $\dot{\omega} \in C^{\ell-1,a}(\Lambda^{1,1}_{\R}(X))$ satisfying the last equality. Thus $\eL$ surjects onto $\Y$. Since all of the entries in $\eL$ are bounded operators, it follows that $\eL$ is bounded as well.
\end{proof}

We would like to show that the kernel of $\eL$ splits $T_{P_0} \mathcal{V}^{\ell,a} = \X$. First we will show that $\eL$ admits a bounded right inverse when $\ed^{\dagger}_0 \rho_0 = 0$.

\begin{lemma}\label{lem-right-inverse}
If $\ed^{\dagger}_0 \rho_0 = 0$, then the map $\eL : \X \to \Y$ admits a bounded right inverse.
\end{lemma}

\begin{proof}
By Lemma~\ref{lem-ellip}, $\ker \eL_1$ and $\coker \eL_1$ are finite-dimensional. Let $A$ denote a closed subspace of $\X_1$ such that $\X_1 = A \oplus \ker \eL_1$, and let $y_1, \ldots, y_k \in \Y$ be elements such that $\Y = \im \eL_1 \oplus \mathrm{span} \{y_1, \ldots, y_k\}$.

By surjectivity of $\eL$ (Lemma~\ref{lem-transversality}), there exist elements $x_1, \ldots, x_k \in A$ such that $\eL_2 x_i = y_i$ for $i=1,\ldots,k$. Each $y \in \Y$ can be expressed uniquely as
\[
y = \eL_1 a + \sum\limits_{i=1}^k c_i y_i, \hspace{0.5cm} c_1, \ldots, c_k \in \R, a \in A.
\]
We can construct a map $h: \Y \to \X$ by
\[
h: y \mapsto a + \sum\limits_{i=1}^k c_i x_i.
\]
which is a right inverse of $\eL$ by construction. The map $h$ is bounded since
\begin{align*}
\eL_1^{-1}: \im \eL_1 \to \X_1 / \ker \eL_1
\end{align*}
is bounded by the Open Mapping Theorem, and the remaining part
\[
\sum\limits_{i=1}^k c_i y_i \mapsto \sum\limits_{i=1}^k c_i x_i
\]
is a linear map between finite-dimensional vector spaces.
\end{proof}

\begin{lemma}
If $\ed^{\dagger}_0 \rho_0 = 0$, then the kernel of $\eL : \X \to \Y$ splits $\X$.
\end{lemma}

\begin{proof}
We can write any $x \in \X$ as
\[
x = (x - h(\eL x)) + h(\eL x),
\]
where $h$ is the right inverse of $\eL$ from Lemma~\ref{lem-right-inverse}. We note that $x - h(\eL x) \in \ker \eL$ and $h(\eL x) \in \im h$. Thus $\X = \ker \eL + \im h$. To see that this sum is direct, suppose that $x \in \ker \eL \cap \im h$. Then we have $\eL x = 0$ and $x = hy$ for some $y \in \Y$, so that $\eL h y = 0$. Since $h$ is a right inverse of $\eL$, we have $y=0$, and we conclude that $x = 0$. Thus $\X = \ker \eL \oplus \im h$, a direct sum of closed subspaces.
\end{proof}

Now recalling the discussion before Lemma~\ref{lem-transversality}, we have the following:

\begin{lemma}
The local universal moduli space $\M(\mathcal{V}^{\ell,a}) = s^{-1}(0)$ is a Banach manifold.
\end{lemma}

\subsection{Transversality}\label{subsec-transvers}

In this section, we would like to adopt a slightly different perspective on the situation, assigning to each $\omega \in \Herm^{\ell-1,a}(X)$ a map $s_{\omega}$ given by $s_{\omega}(\alpha, \rho) = s(\alpha, \rho, \omega)$. Using the Sard--Smale technique, we will show that the local moduli space $\M_{\omega}(\mathcal{V}^{\ell,a}) = s_{\omega}^{-1}(0)$ is a $0$-dimensional manifold for generic $\omega \in \Herm^{\ell-1,a}(X)$.

%
%

%

Recall that $\M(\mathcal{V}^{\ell,a}) := \news^{-1}(0)$. Now, consider the projection
\begin{align*}
&\mathrm{proj}_3 : \mathcal{V}^{\ell,a} \to \Herm^{\ell-1,a}(X) \\
&(\alpha, \rho, \omega) \mapsto \omega,
\end{align*}
and let $\pi: \M(\mathcal{V}^{\ell,a}) \to \Herm^{\ell-1,a}(X)$ denote the restriction of $\mathrm{proj}_3$ to $\M(\mathcal{V}^{\ell,a})$. Once again letting $\eL = (\mathrm{D}\news)_{P_0}$, we note that the tangent space of $\M(\mathcal{V}^{\ell,a})$ at $P_0$ is given by the closed subspace
\[
K = T_{P_0} \M(\mathcal{V}^{\ell,a}) = \ker \eL. 
\]
Thus, the derivative $\mathrm{D}\pi$ is a map $K \to C^{\ell-1,a}(\Lambda^{1,1}_{\R}(X))$. The following lemma is based on Lemma~A.3.6 of~\cite{McDuffSalamon}:

\begin{lemma}\label{lem-ker-coker}
The maps $\mathrm{D}\pi$ and $\eL_1$ are related by the following:
\begin{enumerate}[label=(\roman*)]
\item $\ker \mathrm{D}\pi \cong \ker \eL_1$
\item $\coker \mathrm{D}\pi \cong \coker \eL_1$
\end{enumerate}
\end{lemma}

\begin{proof}
To show (i), note that $x=\alpha + \dot\rho + \dot\omega \in K$ means that
\[
\eL x = \eL_1(\alpha + \dot\rho) + \eL_2 \dot\omega = 0.
\]
So $\mathrm{D}\pi(x) = \dot\omega = 0$ occurs precisely when $x= \alpha + \dot\rho$ with $\eL_1(\alpha + \dot\rho) = 0$.

In order to show (ii), consider the mapping 
 \[
 \Phi: \coker \mathrm{D}\pi \to \coker \eL_1
 \]
 so that for any $\dot\omega \in C^{\ell-1,a}(\Lambda^{1,1}_{\R}(X))$, we have
\[
\Phi \left( \left[\dot\omega \right] \right) = \left[\eL_2 \dot\omega \right].
\]

If $\dot\omega, \dot\omega' \in C^{\ell-1,a}(\Lambda^{1,1}_{\R}(X))$ represent the same equivalence class in $\coker \mathrm{D}\pi$, then for some $\alpha + \dot\rho$ we have $\alpha + \dot\rho + \dot\omega - \dot\omega' \in K$, so that
\[
\eL_1(\alpha + \dot\rho) + \eL_2 (\dot\omega - \dot\omega') = 0.
\]
Thus $\eL_2 \dot\omega$ and $\eL_2 \dot\omega'$ differ by an element of the image of $ \eL_1$, and we conclude that $\Phi$ is well defined.
 
To see that $\Phi$ is injective, suppose that $\left[\eL_2\dot\omega\right] = 0 \in \coker \eL_1$. This means that $\eL_2\dot\omega$ is in the image of $\eL_1$, so that for some $\alpha + \dot\rho$, we have 
\[
\eL_2\dot\omega = \eL_1(\alpha + \dot\rho)
\]
This implies that
 \[
 \eL(-\alpha - \dot\rho + \dot\omega) = 0.
 \]
 In other words, $-\alpha - \dot\rho + \dot\omega \in K$, and since we have $\dot\omega = \mathrm{D}\pi(-\alpha - \dot\rho + \dot\omega)$, it follows that $\dot\omega$ is in the image of $\mathrm{D}\pi$, so that $\left[\dot\omega\right] = 0 \in \coker \mathrm{D}\pi$.

Now let us show that $\Phi$ is surjective. Since $\eL$ is surjective, for any $y \in \Y$, there exists $\alpha+\dot\rho + \dot\omega$ such that $\eL(\alpha+\dot\rho + \dot\omega) = \eL_1(\alpha+\dot\rho) + \eL_2\dot\omega = y$. The class $[\dot\omega]$ will be sent to the class $[y]$, establishing the claim.
\end{proof}

Lemma~\ref{lem-ellip} and Lemma~\ref{lem-ker-coker} imply:
\begin{corollary}\label{cor-pi-fredholm}
The map $\pi: \M(\mathcal{V}^{\ell,a}) \to \Herm^{\ell,a}(X)$ is Fredholm, and its Fredholm index is equal to $0$.
\end{corollary}

We are now ready to prove Theorem~\ref{thm-main-loc-holder}:

\begin{proof}[Proof of Theorem~\ref{thm-main-loc-holder}]
By Corollary~\ref{cor-pi-fredholm} and the Sard--Smale Theorem, the set $\Herm_{\mathrm{reg}, \mathcal{V}^{\ell-1,a}}^{\ell,a}(X)$ of regular values of the map
\[
\pi: \M(\mathcal{V}^{\ell,a}) \to \Herm^{\ell-1,a}(X)
\]
is comeagre in $\Herm^{\ell-1,a}(X)$. Let $\omega \in \Herm^{\ell-1,a}_{\mathrm{reg},\mathcal{V}^{\ell,a}}(X)$ and let $P_0 \in \M_{\omega}(\mathcal{V}^{\ell,a})$. The fact that $\omega \in \Herm_{\mathrm{reg},\mathcal{V}^{\ell-1,a}}^{\ell,a}(X)$ means that $\coker (\mathrm{D}\pi)_{P_0} = 0$, so that $\coker (\mathrm{D}\news_{\omega})_{P_0} = 0$ by Lemma~\ref{lem-ker-coker}. I.e. $(\mathrm{D}s_{\omega})_{P_0}$ is surjective. Since $P_0 \in \M_{\omega}(\mathcal{V}^{\ell,a})$ is arbitrary, it follows by the regular value theorem that $\M_{\omega}(\mathcal{V}^{\ell,a})$ is an embedded submanifold of $\M(\mathcal{V}^{\ell,a})$, whose dimension is
\[
\dim \M_{\omega}(\mathcal{V}^{\ell,a}) = \ind(\mathrm{D}\news_{\omega}) = \ind(\eL_1) = 0.
\]
This means that for all $\omega \in \Herm_{\mathrm{reg},\mathcal{V}^{\ell,a}}^{\ell-1,a}(X)$, the set $\M_{\omega}(\mathcal{V}^{\ell,a})$ consists of isolated points.
\end{proof}

\subsection{Regularity of PSL submanifolds}

In this section, we prove an elliptic regularity theorem for PSL submanifolds, which we will need in the following section in order to upgrade our local result to allow for smooth parameters.

\begin{theorem}\label{thm-ellip-reg}
Suppose that $f: (L, \rho) \to (X, J ,\omega, \Omega)$ is PSL, where $f$ and $\rho$ are $C^{1,a}$, and $\omega$ is a $C^{\infty}$ metric on $X$. Then $f$ and $\rho$ are also $C^{\infty}$.
\end{theorem}

\begin{proof}
We will prove this by elliptic bootstrapping. Thus, assume that $f \in C^{k,a}(L,X)$ and $\rho \in C^{k,a}(\Hcal^3(L)^{\perp})$, where $k \geq 1$. We will show that $f \in C^{k+1,a}(L,X)$ and $\rho \in C^{k+1,a}(\Hcal^3(L)^{\perp})$.

Let $p \in L$ and let $U \subset L$ be a coordinate patch containing $p$ that is small enough so that $f(U) \subset X$ is graphical over $T_p L$. Let $\tilde{U} \subset X$ be a neighbourhood of $p$ in $X$ such that $U = \tilde{U} \cap L$, and the exponential map of $X$ is a diffeomorphism when restricted to $\tilde{U}$.

We will identify $T_p L \hookrightarrow T_{f(p)} X$ with the standard totally real $\R^3 \hookrightarrow \C^3 \cong \R^3 \oplus i\R^3$, and $J(T_p L)$ with the complementary subspace $i\R^3$. Also identify $U \subset L$ with its pre-image $(\exp^L_p)^{-1}(U)$ under the exponential map on $L$, a subset of $\R^3$. By assumption, the set $(\exp^X_{f(p)})^{-1}(f(U))$ is then the graph of a $C^{k,a}$ function $\phi: U \to i\R^3$. Now, $\R^3$ is PSL in $\C^3$ with $\rho = 0$, so we can work in the chart $\mathcal{V}^{k,a}$ as in (\ref{def-chart-holder}) centered at $U \subset \R^3 \subset \C^3$, and we identify $\phi$ with a small $C^{k,a}$ $1$-form $\alpha$ on $U$ as in (\ref{eqn-def-def}), so that
\[
f_{\alpha}(x) = (x,i\alpha^{\sharp}(x)) = (x,i\phi(x)) \in \R^3 \oplus i\R^3, \hspace{0.5cm} x \in U.
\]

Now, knowing that $s_{\omega}(0, 0) = 0$, we write down the Taylor expansion of $s_{\omega}$ about $(0,0)$ and evaluate at $f: L \to X$ to obtain
\begin{equation}\label{eqn-taylor-exp}
0=s_{\omega}(\alpha,\rho) = (\mathrm{D} s_{\omega})_{(0,0)} (\alpha + \rho) + o(|\alpha|,|\rho|).
\end{equation}


Let $\alpha_U$ and $\rho_U$ denote the (vector valued) component functions of $\alpha$ and $\rho$ in local coordinates on $U$, respectively. We apply the coordinate derivative $\partial_{x^j}$ to both sides of (\ref{eqn-taylor-exp}), and obtain
\begin{align*}
0 =& \left(\partial_{x^j} \eL_1^{U}\right) (\alpha_{U} + \rho_{U}) + \eL_1^{U} \partial_{x^j}(\alpha_{U} + \rho_{U}) \\
& \hspace{2cm}+ R(\alpha_{U}, \rho_{U}, \partial_{x^j} \alpha_{U}, \partial_{x^j} \rho_{U}),
\end{align*}
where $\eL_1^{U}$ denotes the matrix function representing $(\mathrm{D} s_{\omega})_{(0,0)}$ in local coordinates on $U$, and $R(\alpha_{U}, \rho_{U}, \partial_{x^j} \alpha_{U}, \partial_{x^j} \rho_{U})$ is some remainder term.


Rearranging, we have
\[
\eL_1^{U} (\partial_{x^j}\alpha_{U} + \partial_{x^j}\rho_{U}) = R'(\alpha_{U}, \rho_{U}, \partial_{x^j} \alpha_{U}, \partial_{x^j} \rho_{U}).
\]
Since $\alpha$ and $\rho$ are assumed to be $C^{k,a}$, we see that $R'(\alpha_{U}, \rho_{U},\partial_{x^j} \alpha_{U}, \partial_{x^j} \rho_{U})$ is $C^{k-1,a}$. Now Lemma~\ref{lem-ellip} tells us that $\eL_1^{U}$ is a first order elliptic operator, so by standard Schauder estimates, we have
\begin{align*}
\norm{\partial_{x^j}\alpha_{U} + \partial_{x^j}\rho_{U}}_{C^{k,a}} &\leq C \left(\norm{R'(\alpha_{U}, \rho_{U}, \partial_{x^j} \alpha_{U}, \partial_{x^j} \rho_{U})}_{C^{k-1,a}}\right. \\
&\hspace{2cm}\left.+ \norm{\partial_{x^j}\alpha_{U} + \partial_{x^j}\rho_{U}}_{C^0}\right) < \infty.
\end{align*}
We conclude that $\partial_{x^j} \alpha_{U} + \partial_{x^j} \rho_{U}$ is $C^{k,a}$, and since this is true for each partial derivative $\partial_{x^j}$, we see that $\alpha_{U}$ and $\rho_{U}$ are $C^{k+1,a}$.
\end{proof}

\subsection{Smooth parameters and the Taubes Trick}\label{subsec-taubes}

We will adapt an argument of Taubes (\emph{the Taubes trick}) to extend the local result of the previous section to allow for smooth parameters. We may speak of the \textbf{smooth parameter space} $\B$ given by
\[
\B := \bigcup_{\underset{\omega \in \Herm(X)}{f \in B}} \left(\left\{[f]\right\} \times \left(\T_{[f],\omega} \cap \mathcal{H}^3(L)^{\perp_{\omega}}\right) \times \{\omega\}\right), 
\]
where the lack of superscripts indicates that elements are smooth, and we are equipping $\B$ with the $C^{\infty}$ topology. Here $B$ is once again the base space in the $\Diff(L)$-principal bundle $\Emb(L,X) \to B$ described in Section~\ref{subsec-moduli-space}, but this time with smooth data and equipped with the $C^{\infty}$ topology. Furthermore $\Herm(X)$ is the space of smooth Hermitian metrics on $X$, also with the $C^{\infty}$ topology. We can once again use a local chart $\mathcal{V}$ similar to (\ref{def-chart-holder}), where all the data used to construct $\mathcal{V}$ is smooth:
\begin{equation}\label{def-chart-smooth}
\mathcal{V} := \bigcup_{\underset{\omega \in \Herm(X)}{[f] \in \mathcal{U}}} \left(\{[f]\} \times \left(\T_{[f],\omega} \cap \mathcal{H}^3(L)^{\perp}\right) \times \{\omega\}\right).
\end{equation}
Furthermore we have a section
\[
s: \mathcal{V} \to \Y = C^{\infty}(\Lambda^2(L)) \oplus \ed^{\dagger}_0C^{\infty}(\Lambda^1(L))
\]
defined just as in (\ref{eqn-sec-def}), and we write $\M_{\omega}(\mathcal{V}) = s^{-1}(0)$ for the local moduli space with respect to $\omega \in \Herm(X)$.

 We will prove the following:

\begin{proposition}[Local version of Theorem~\ref{thm-main} with smooth parameters]\label{prop-main-loc-smooth}
There is a comeagre set of Hermitian metrics $\Herm_{\reg,\mathcal{V}}(X) \subset \Herm(X)$ such that for all $\omega \in \Herm_{\reg,\mathcal{V}}(X)$, the set $\M_{\omega}(\mathcal{V})$ consists of isolated points.
\end{proposition}

 Our use of the Taubes trick follows Windes \cite{Windes}, relying in a key way on Breuning's compactness theorem \cite{Breuning} for immersions whose second fundamental form is bounded in $L^p$.

 Let $\omega_0$ be a smooth reference metric on $X$. Given $K > 0$ and $\omega \in \Herm(X)$, let $\M_{\omega,K}(\mathcal{V})$ (resp. $\M_{\omega,K}^{\ell,a}(\mathcal{V}^{\ell,a})$) denote the set of all $(f,\rho) \in \M_{\omega}(\mathcal{V})$ (resp. $\in \M_{\omega}(\mathcal{V}^{\ell,a})$) such that for all points $x,y \in L$, we have
\begin{equation}\label{eqn-taubes-condition-1}
d_{\omega_0}(f(x),f(y)) \geq \frac{1}{K} \cdot d_{f^*\omega_0}(x,y),
\end{equation}
and the second fundamental form $\two_f$ with respect to $\omega_0$ of $f$ satisfies the bound
\begin{equation}\label{eqn-taubes-condition-2}
|\two_f|_{L^p} \leq K
\end{equation}
for some $p > 3$, where the $L^p$ bound being computed with respect to $\omega_0$.

Now, let $\Herm_{\reg, K}(X)$ (resp. $\Herm_{\reg,K}^{\ell,a}(X)$) denote the space of $\omega \in \Herm(X)$ (resp. $\omega \in \Herm^{\ell,a}(X)$) such that the operator $\mathrm{D}s_{\omega}$ is a submersion at each $(f,\rho) \in \M_{\omega,K}(\mathcal{V})$.

\begin{lemma}
For each $K > 0$, the space $\Herm_{\reg,K}(X)$ (resp. $\Herm_{\reg,K}^{\ell,a}(X)$) is open and dense in $\Herm(X)$ (resp. $\Herm^{\ell,a}(X)$) with respect to the $C^{\infty}$ (resp. $C^{\ell,a}$) topology.
\end{lemma}

\begin{proof}
Fix $K > 0$. We show that $\Herm_{\reg,K}(X)$ is open by showing that its complement is closed. Suppose that $\left\{\omega_{\nu}\right\}_{\nu=1}^{\infty} \subset \Herm(X) \setminus \Herm_{\reg,K}(X)$ is a sequence such that $\omega_{\nu} \to \omega_{\infty} \in \Herm(X)$ in $C^{\infty}$. Then for each $\nu$, there is a pair $(f_{\nu},\rho_{\nu}) \in \M_{\omega_{\nu},K}$ such that $\mathrm{D}s_{\omega_{\nu}}$ is not a submersion at $(f_{\nu},\rho_{\nu})$. By the main result of \cite{Breuning}, there exists a subsequence (which we will also denote $f_{\nu}$), a $C^1$ map $f_{\infty}: L \to X$, and a sequence of diffeomorphisms $\phi_{\nu}$ such that $f_{\nu} \circ \phi_{\nu}$ converges to $f_{\infty}$ in $C^1$. Because of the diffeomorphism invariance of the PSL system shown in Section~\ref{subsec-moduli-space}, we can reparameterize so that $f_{\nu} \to f_{\infty}$ in $C^1$.

We also note that a further subsequence of the $3$-forms $\rho_{\nu}$ also converges to a limit $\rho_{\infty}$. One sees this by noting that for each $\nu$, we have $\rho_{\nu} \in \T_{f_{\nu},\omega_{\nu}}$, and so by the fact that $\omega_{\nu} \to \omega$ and $f_{\nu} \to f$, the $C^{1,a}$ norms of the $3$-forms $\rho_{\nu}$ are uniformly bounded with respect to the limiting metric $\omega_{\infty}$. Thus some subsequence of the $\rho_{\nu}$ converges in $C^1$ to a $3$-form $\rho_{\infty}$. Each pair $(f_{\nu}, \rho_{\nu})$ satisfies the equation
\[
f_{\nu}^* \omega_{\nu} = -\ed^{\dagger}_{f_{\nu}^* \omega_{\nu}} \rho_{\nu},
\]
and since $f_{\nu} \to f_{\infty}$ and $\rho_{\nu} \to \rho_{\infty}$ in $C^1$, we have
\[
f_{\infty}^* \omega_{\infty} = -\ed^{\dagger}_{f_{\infty}^* \omega_{\infty}} \rho_{\infty}.
\]
By elliptic regularity (Theorem~\ref{thm-ellip-reg}) we see that $f_{\infty}$ and $\rho_{\infty}$ are $C^{\infty}$.  The main result in \cite{Breuning} guarantees that $f_{\infty}$ satisfies (\ref{eqn-taubes-condition-1}) and (\ref{eqn-taubes-condition-2}), and we can conclude that $(f_{\infty},\rho_{\infty}) \in \M_{\omega_{\infty},K}$. Being a submersion is an open condition, so the operator $\mathrm{D}s_{\omega}$ is not a submersion at $(f_{\infty},\rho_{\infty})$. Thus $\omega_{\infty} \not\in \Herm_{\reg,K}(X)$. A similar argument shows that $\Herm^{\ell,a}_{\reg, K}(X)$ is open in $\Herm^{\ell,a}(X)$ with respect to the $C^{\ell,a}$ topology for sufficiently large $\ell$.

%

Now we will show that $\Herm_{\reg,K}(X)$ is dense in $\Herm(X)$, following the same argument that appears in Section~3.2 of \cite{McDuffSalamon}. Let $\omega \in \Herm(X)$. We know that $\Herm^{\ell,a}_{\reg}(X)$ is dense in $\Herm^{\ell,a}(X)$ since it is a comeagre set. Therefore there is a sequence $\left\{\omega_{\ell}\right\}_{\ell=\ell_0}^{\infty}$ where for each $\ell$ we have $\omega_{\ell} \in \Herm^{\ell,a}_{\reg}(X)$ and such that $\norm{\omega - \omega_{\ell}}_{C^{\ell,a}} \leq 2^{-\ell}$. Note that for each $\ell$, we have $\omega_{\ell} \in \Herm^{\ell,a}_{\reg,K}(X)$.

Since $\Herm^{\ell,a}_{\reg,K}(X)$ is open in $\Herm^{\ell,a}(X)$, for each $\ell$ there is an $\epsilon_{\ell} > 0$ such that whenever $\omega' \in \Herm^{\ell,a}(X)$ and $\norm{\omega' - \omega_{\ell}}_{C^{\ell,a}} < \epsilon_{\ell}$, we have $\omega' \in \Herm^{\ell,a}_{\reg,K}(X)$. Now choose smooth Hermitian metrics $\omega'_{\ell} \in \Herm(X)$ such that
\[
\norm{\omega'_{\ell} - \omega_{\ell}}_{C^{\ell,a}} < \min\left\{\epsilon_{\ell}, 2^{-\ell}\right\}.
\]
Then $\omega'_{\ell} \in \Herm(X) \cap \Herm^{\ell,a}_{\reg,K}(X) = \Herm_{\reg,K}(X)$, and $\omega'_{\ell} \to \omega$ as $\ell \to \infty$. Thus $\Herm_{\reg,K}(X)$ is dense in $\Herm(X)$, completing the proof.
\end{proof}

To complete the proof of Proposition~\ref{prop-main-loc-smooth}, let
\[
\Herm_{\reg, \mathcal{V}}(X) = \bigcap_{K \in \mathbb{N}} \Herm_{\reg, K}(X).
\]
Then $\Herm_{\reg, \mathcal{V}}(X)$ is an intersection of countably many dense open sets in $\Herm(X)$, and so it is comeagre. If $\omega \in \Herm_{\reg}(X)$, then the operator $\mathrm{D}s_{\omega}$ is a submersion at any smooth $(f,\rho)$, so that $\M_{\omega}(\mathcal{V})$ consists of isolated points.

\subsection{Returning to the global picture}\label{subsec-return-glob}

Finally, we can prove Theorem~\ref{thm-main}.
%


\begin{proof}[Proof of Theorem~\ref{thm-main}]
Observe that each chart $\mathcal{V}$ for $\B$ of the form (\ref{def-chart-smooth}) admits a natural projection $\mathcal{V} \to \mathcal{U}$ to the chart $\mathcal{U}$ of $B$. Since $B$ is a separable space, there are countably many charts $\mathcal{V}_i$ of the form (\ref{def-chart-smooth}) such that
\[
\bigcup\limits_{i=1}^{\infty} \mathcal{V}_i = \B.
\]
For each $\mathcal{V}_i$, Proposition~\ref{prop-main-loc-smooth} implies that there is a comeagre set $\Herm_{\reg,\mathcal{V}_i}(X) \subset \Herm(X)$ such that if $\omega \in \Herm_{\reg,\mathcal{V}_i}(X)$, then $\M_{\omega}(\mathcal{V}_i)$ consists of isolated points. Define
\[
\Herm_{\reg}(X) = \bigcap\limits_{i=1}^{\infty} \Herm_{\reg,\mathcal{V}_i}(X).
\]
If $\omega \in \Herm_{\reg}(X)$, then $\M_{\omega}(\mathcal{V}_i)$ consists of isolated points for every $i \in \mathbb{N}$. Thus the global moduli space $\M_{\omega}$ consists of isolated points. Furthermore the set $\Herm_{\reg}(X) \subset \Herm(X)$ is comeagre, since it is a countable intersection of comeagre sets.
\end{proof}

\bibliographystyle{acm}
\bibliography{refs}

\end{document}